\documentclass[11pt]{article}
\usepackage{amsfonts,amssymb,amsmath,amsthm,latexsym}
\usepackage{verbatim}
\usepackage{tikz}
\usepackage{float}
\usepackage{caption}

\title{\vspace{-0.6cm} A problem of Erd\H{o}s on the minimum number of $k$-cliques}
\author{
Shagnik Das\thanks{Department of Mathematics, UCLA, Los
Angeles, CA 90095. Email: {\tt shagnik@ucla.edu}.}
\and
Hao Huang\thanks{Department of Mathematics, UCLA, Los
Angeles, CA 90095. Email: {\tt huanghao@math.ucla.edu}. Research supported by a
UC Dissertation Year Fellowship.}
\and
Jie Ma\thanks{Department of Mathematics, UCLA, Los
Angeles, CA 90095. Email: {\tt jiema@math.ucla.edu}.}
\and
Humberto Naves\thanks{Department of Mathematics, UCLA, Los
Angeles, CA 90095. Email: {\tt hnaves@math.ucla.edu}.}
\and
Benny Sudakov\thanks{Department of Mathematics, UCLA, Los Angeles, CA 90095. Email:
{\tt bsudakov@math.ucla.edu}. Research supported in part by NSF grant DMS-1101185,
NSF CAREER award DMS-0812005, and by a USA-Israeli BSF grant.}}

\date{}

\newtheorem{thm}{Theorem}[section]
\newtheorem{prop}[thm]{Proposition}
\newtheorem{lemma}[thm]{Lemma}
\newtheorem{cor}[thm]{Corollary}

\newtheorem{fact}[thm]{Fact}

\def\qed{\ifvmode\mbox{ }\else\unskip\fi\hskip 1em plus 10fill$\Box$}

\begin{document}
\maketitle

\begin{abstract}
Fifty years ago Erd\H{o}s asked to determine the minimum number of $k$-cliques in a graph on $n$ vertices with independence number less than $l$. He conjectured that this minimum is achieved by the disjoint union of $l-1$ complete graphs of size $\frac{n}{l-1}$. This conjecture was disproved by Nikiforov who showed that the balanced blow-up of a $5$-cycle has fewer $4$-cliques than the union of $2$ complete graphs of size $\frac{n}{2}$.

In this paper we solve Erd\H{o}s' problem for $(k,l)=(3,4)$ and $(k,l)=(4,3)$. Using stability arguments we also characterize the precise structure of extremal examples, confirming Erd\H{o}s' conjecture for $(k,l)=(3,4)$ and showing that a blow-up of a $5$-cycle gives the minimum for $(k,l)=(4,3)$.
\end{abstract}

\section{Introduction} \label{intro}

Let $K_l$ denote a complete graph on $l$ vertices and let $\overline{K_l}$ be its complement, i.e., an independent set of size $l$. One of the central results in extremal combinatorics is Tur\'{a}n's theorem \cite{turan}, which asserts that the maximum number of edges in a $K_l$-free graph on $n$ vertices is attained by the Tur\'{a}n graph $T_{n,l-1}$, a complete $(l-1)$-partite graph with nearly-equal parts.  This theorem has since been extended and generalized in many different ways.  Since an edge can be thought of as a clique on $2$ vertices, a natural generalization is to ask for the maximum number of $K_k$ in an $n$-vertex graph with no $K_l$.  Zykov \cite{zykov} showed that this maximum was also attained by the Tur\'{a}n graph $T_{n,l-1}$.

For any integers $k, l \ge 2$ and $n$, we define $f(n,k,l)$ to be the minimum number of copies of $K_k$ in a $\overline{K_l}$-free graph on $n$ vertices.  If one takes the complements of the graphs in Tur\'{a}n's theorem, then the theorem gives the minimum number of edges in an $n$-vertex $\overline{K_l}$-free graph.  Thus the question of determining $f(n,k,l)$ is precisely the Zykov-type generalization of this complementary version.  Fifty years ago Erd\H{o}s \cite{erdos} asked to determine $f(n,k,l)$ and conjectured that the minimum is given by the complement of the Tur\'{a}n graph, $\overline{T_{n,l-1}}$, which is the disjoint union of $l-1$ complete graphs of equal size.  When $k=2$, this follows from Tur\'{a}n's theorem.

Note that a graph is $\overline{K_l}$-free precisely when its independence number is less than $l$.  One can thus also view this problem as a strengthening of Ramsey's theorem, which states that any sufficiently large graph either has a clique of size $k$ or an independent set of size $l$.  The $(k,l)$-problem asks how many cliques of size $k$ a graph must have when its independence number is less than $l$.

\medskip

Lorden \cite{lorden} proved Erd\H{o}s' conjecture to be true for the $(3,3)$-case by a simple double-counting argument.  However, no further progress was made in the next forty years, until Nikiforov \cite{nikiforov} disproved the conjecture in the case $(4,3)$ by showing the balanced blow-up of $C_5$, which is $\overline{K_3}$-free, contains fewer $4$-cliques than the disjoint union of two cliques, $\overline{T_{n,2}}$.  In a \emph{blow-up} of a graph, we replace every vertex with a clique, and every edge with a complete bipartite graph.  We say the blow-up is \emph{balanced} if the cliques are all of the same size.  In a subsequent preprint \cite{nikiforov_pre}, Nikiforov showed that his construction is optimal under the additional assumption that the graph should be nearly-regular.

Moreover, by considering blow-ups of Ramsey graphs, Nikiforov showed that the conjecture could only hold for finitely many $(k,l)$ when $k,l \ge 3$.  In particular, he conjectured that equality held only for the cases $(3,3)$ and $(3,4)$, the latter of which remained an open problem.

\subsection{Our results}

In this paper, we first sharpen Nikiforov's result by showing that Erd\H{o}s' conjecture is always false when $k \ge 4$ and $l \ge 3$, or when $k = 3$ and $l \ge 2074$.  We obtain these results through a combination of explicit and random counterexamples.

We then solve the problem in the cases $(k,l) = (4,3)$ and $(3,4)$.
Using the machinery of flag algebras developed by Razborov
\cite{raz_flag}, we are able to obtain the asymptotic values of
$f(n,4,3)$ and $f(n,3,4)$.  By analyzing the corresponding
semi-definite programming solutions, we are then able to derive
stability results for these cases, which in turn allow us to
determine $f(n,4,3)$ and $f(n,3,4)$ exactly for large $n$, and also
to characterize the extremal graphs. In particular, we show that a
blow-up of $C_5$ is indeed optimal for the $(4,3)$ problem, while
Erd\H{o}s' conjecture holds for the $(3,4)$ problem.  Our results
are summarized in the following theorems.

\begin{thm} \label{thm43}
$f(n,4,3) = \frac{3}{25} \binom{n}{4} + O(n^3)$, where the minimum is achieved by a blow-up of $C_5$ with five parts of roughly equal sizes.  Moreover, the extremal structure is unique for sufficiently large $n$.
\end{thm}

\noindent We determine the exact sizes of the parts of the blow-up by solving an integer optimization problem, the precise results of which are given in Section \ref{sec43}.

\begin{thm} \label{thm34}
$f(n,3,4) = \binom{\lfloor n/3 \rfloor}{3}+\binom{\lfloor (n+1)/3 \rfloor}{3}+\binom{\lfloor (n+2)/3 \rfloor}{3} \sim \frac{1}{9} \binom{n}{3}$, where for large $n$ the minimum is achieved by three disjoint cliques that are as equal as possible.  Moreover, any extremal graph must be spanned by three such cliques.
\end{thm}

\noindent Note that in this case the extremal graph is not unique, as we may have partial matchings between the cliques without introducing any extra triangles.

As we remark in our concluding section, solutions of corresponding SDP problems strongly suggest that a disjoint union of cliques remains optimal for the $(3,5)$- and $(3,6)$-problems, contrary to Nikiforov's conjecture.

\subsection{Notation and organization}

Given a graph $G$ on vertices $V(G)$, and a vertex $v \in V(G)$, we denote by $N(v)$ the set of neighbors of $v$ in $G$, and by $\overline{N(v)}$ the set of non-neighbors of $v$.  The complement graph $\overline{G}$ shares the same vertices as $G$, and has an edge $\{u, v\}$ if and only if $\{ u, v \}$ is not an edge of $G$.  We denote the independence number of $G$ by $\alpha(G)$.  The complete graph on $k$ vertices is denoted by $K_k$.  In particular, a graph $G$ is $\overline{K_l}$-free if and only if $\alpha(G) < l$.  Some other graphs we will use are the cycles $C_k$, and paths $P_k$ where in each case the subscript refers to the number of edges.

Given a fixed graph $H$, for any graph $G$ we let $t_{H}(G)$ denoted the number of induced copies of $H$ in $G$. In the case $H = K_k$, we simplify the notation to $t_k(G)$.  Using this notation, we can define
\[ f(n,k,l) = \min \{ t_k(G) : |V(G)| = n, t_l( \overline{G} ) = 0 \}. \]

\medskip

The rest of the paper is organized as follows.  In the next section, we construct counterexamples to Erd\H{o}s' conjecture in the case $k \ge 4$ and $l \ge 3$ or $k = 3$ and $l$ large.  In Section \ref{flag_intro}, we provide an informal introduction to our main tool, flag algebras.  Sections \ref{sec43} and \ref{sec34} contain the proofs of our main results for the $(4,3)$- and $(3,4)$-problems respectively.  The final section contains some concluding remarks and open problems.

Some technical details are given in the appendices: Appendix \ref{app_impflags} provides some remarks regarding implementation of flag algebras, and Appendix \ref{app_intopt} contains the proof of the integer optimization result for the $(4,3)$-problem.

\section{Counterexamples to Erd\H{o}s' conjecture} \label{nonturan}

Nikiforov \cite{nikiforov} showed that not only was Erd\H{o}s' conjecture not true in general, but that it held only finitely often.  He used bounds on the Ramsey numbers $R(3,l)$ to show the existence of $k_0$ and $l_0$ such that whenever $k > k_0$ or $l > l_0$, blow-ups of Ramsey graphs did better than disjoint unions of cliques $\overline{T_{n,l-1}}$.  In the following theorem, we use a combination of explicit and random constructions to further improve this result.

\begin{thm} \label{counterexample}
$\overline{T_{n,l-1}}$ is not optimal for the $(k,l)$-problem when
\begin{itemize}
    \item[(i)] $k \ge 4$ and $l \ge 3$, or
    \item[(ii)] $k = 3$ and $l \ge 2074$.
\end{itemize}
\end{thm}

\subsection{The $(k,l)$-problem with $k \ge 4$}

Let us first consider the case $l=3$.  That is, we are looking to minimize the number of $k$-cliques in a graph with independence number at most $2$.  For the $(4,3)$-problem, Nikiforov \cite{nikiforov_pre} gave an explicit counter-example to Erd\H{o}s's conjecture by showing that a blow-up of $C_5$ contains fewer triangles than the graph $\overline{T_{n,2}}$, which consists of two disjoint cliques.  In fact, it is easy to see that this construction is better than $\overline{T_{n,2}}$ for any $k \ge 4$.  Indeed, a disjoint union of two cliques contains, asymptotically, $2 \binom{\frac{n}{2}}{k} \sim \frac{1}{2^{k-1}} \binom{n}{k}$ $k$-cliques.  On the other hand, the blow-up of $C_5$ contains $5 \left( \binom{ \frac{2n}{5} }{k} - \binom{ \frac{n}{5} }{k} \right) \sim \frac{2^k - 1}{5^{k-1}} \binom{n}{k}$ $k$-cliques.  For $k \ge 4$, we have $\frac{2^k - 1}{5^{k-1}} < \frac{1}{2^{k-1}}$, and so $\overline{T_{n,2}}$ is asymptotically not optimal for the $(k,3)$-problem.

For $l \ge 4$, the graph $\overline{T_{n,l-1}}$ consists of $l-1$ disjoint cliques.  However, as shown above, if we replace two of these cliques with a blow-up of $C_5$ on the same number of vertices, we will reduce the number of $k$-cliques. Formally, this construction has a blow-up of $C_5$ on five parts of size $\frac{2n}{5(l-1)}$, and $l-3$ disjoint cliques of size $\frac{n}{l-1}$, and contains fewer $k$-cliques than $\overline{T_{n,l-1}}$.  This shows that a disjoint union of cliques is not optimal for the $(k,l)$-problem for any $k \ge 4$ and $l \ge 3$.

\subsection{The $(3,l)$-problem}

The situation is quite different when $k=3$.  As we will show later, the disjoint union of cliques is optimal for the $(3,3)$- and $(3,4)$-problems.  However, unlike the case $k = 2$, this construction ceases to be optimal for large values of $l$. We consider the random graph $G \sim G(m,p)$ on $m$ vertices, with every edge appearing independently with probability $p$.  For suitable parameters $l, m,$ and $p$, we show that with positive probability the balanced blow-up of $G$ has no independent set of size $l$ and has fewer triangles than $\overline{T_{n,l-1}}$.  First we count the number of triangles in a balanced blow-up of an $m$-vertex graph $G$ to $n$ vertices.

There are three ways to obtain a triangle in the blow-up.  The
vertices of the triangle can all come from one part, in which case
there are $\frac{n}{m}$ vertices to choose from.  As there are $m$
vertices in $G$, there are $m \binom{ \frac{n}{m} }{3} \sim
\frac{1}{m^2} \binom{n}{3}$ such triangles.  Alternatively, the
vertices of the triangle can come from an edge in $G$, with two
vertices from one part, and the third vertex from the other.  There
are two ways to split the vertices, and $e(G)$ edges, so the total
number of such triangles is $2 e(G) \binom{ \frac{n}{m} }{2} \binom{
\frac{n}{m} }{1} \sim \frac{6 e(G)}{m^3} \binom{n}{3}$.  Finally,
the vertices of the triangle can come from a triangle in $G$, with
one vertex from each of the three parts.  There are $t_3(G)$
triangles in $G$, and so the number of such triangles is $t_3(G)
\left( \frac{n}{m} \right)^3 \sim \frac{6 t_3(G)}{m^3}
\binom{n}{3}$.  Thus the total number of triangles in the blow-up of
$G$ is asymptotically $\left( \frac{6 \left( e(G) + t_3(G) \right)
}{m^3} + \frac{1}{m^2} \right) \binom{n}{3}$.

On the other hand, $\overline{T_{n,l-1}}$ has $(l-1) \binom{
\frac{n}{l-1} }{3} \sim \frac{1}{(l-1)^2} \binom{n}{3}$ triangles.
Thus to obtain a counter-example to Erd\H{o}s's conjecture, we need
to show that for some $l,m$ and $p$, with positive probability the random graph $G
\sim G(m,p)$ has no independent set of size $l$ and $\frac{6 \left(
e(G) + t_3(G) \right) }{m^3} + \frac{1}{m^2} < \frac{1}{(l-1)^2}$,
or $e(G) + t_3(G) < \frac{m^3}{6(l-1)^2} - \frac{m}{6}$.  Let us
call such a graph `suitable'.

Let $B_1$ be the event that $\alpha(G) \ge l$, where $\alpha(G)$ is
the independence number of $G$.  For some parameters $s$ and $t$,
let $B_2$ be the event $\{ e(G) - \mathbb{E}[e(G)] \ge s \}$, and
$B_3$ the event $\{ t_3(G) - \mathbb{E}[t_3(G)] \ge t \}$.  If
$\mathbb{E}[e(G) + t_3(G)] + s + t \le \frac{m^3}{6(l-1)^2} -
\frac{m}{6}$, then $\left\{ e(G) + t_3(G) \ge \frac{m^3}{6(l-1)^2} -
\frac{m}{6} \right\} \subset B_2 \cup B_3$. Then we have
\[ \mathbb{P}( G \textrm{ not suitable}) \le \mathbb{P}( B_1 \cup B_2 \cup B_3 ) \le \mathbb{P}( B_1 ) + \mathbb{P}( B_2 \cup B_3 ). \]

We use a union bound for $B_1$: there are $\binom{m}{l}$ sets of $l$
vertices, and the probability that a given set has no edges is
$(1-p)^{\binom{l}{2}}$.  Using the bound $\binom{n}{r} \le \left(
\frac{ne}{r} \right)^r$, we have
\[ \mathbb{P}(B_1) \le \binom{m}{l} (1 - p)^{ \binom{l}{2} } \le \left( \frac{m e (1-p)^{\frac{l-1}{2}}}{l} \right)^l. \]

Note that the other two events are increasing; that is, they are preserved by the addition of edges.  It then follows from Kleitman's Lemma (see Chapter 6 in \cite{alonspencer}) that $\mathbb{P}(B_2 \cap B_3) \ge \mathbb{P}(B_2) \mathbb{P}(B_3)$, and so
\[ \mathbb{P}(B_2 \cup B_3) = \mathbb{P}(B_2) + \mathbb{P}(B_3) - \mathbb{P}(B_2 \cap B_3) \le \mathbb{P}(B_2) + \mathbb{P}(B_3) - \mathbb{P}(B_2)\mathbb{P}(B_3) = \mathbb{P}(B_2) + \mathbb{P}(B_3) \left( 1 - \mathbb{P}(B_2) \right). \]
Moreover, since the right-hand side is increasing in both $\mathbb{P}(B_2)$ and $\mathbb{P}(B_3)$, we can replace the probabilities with upper bounds to obtain an upper bound on $\mathbb{P}(B_2 \cup B_3)$.  To obtain these upper bounds, we use the following second moment concentration inequality from \cite{alonspencer}:

\begin{prop} \label{variance}
Let $X$ be a random variable with expectation $\mathbb{E}[X] = \mu$ and variance $\sigma^2$.  Then for all $\lambda > 0$,
\[ \mathbb{P}( X - \mu \ge \lambda ) \le \frac{ \sigma^2 }{ \lambda^2 + \sigma^2}. \]
\end{prop}

For the event $B_2$, with $X = e(G)$, we have $X \sim
\mathrm{Bin}\left( \binom{m}{2} , p \right)$, and so $\mu =
\binom{m}{2}p$ and $\sigma^2 = \binom{m}{2} p ( 1 - p )$.  This
gives $\mathbb{P}(B_2) \le \frac{\binom{m}{2} p ( 1 - p) }{s^2 + \binom{m}{2} p ( 1 - p) }$.

For the event $B_3$, let $X = t_3(G)$.  There are $\binom{m}{3}$
possible triangles, each of which appears with probability $p^3$,
and hence $\mu = \binom{m}{3} p^3$.  To find the variance, we note
that any fixed triangle $T$ is independent of all triangles except
those that share at least two vertices with $T$.  A quick
calculation gives $\sigma^2 = \binom{m}{3} p^3 \left[ (1 - p^3) +
3(m-3) p^2 (1 - p) \right]$.  Thus $\mathbb{P}(B_3) \le \frac{\binom{m}{3} p^3 \left[ (1 - p^3) + 3 (m-3) p^2 (1-p) \right]}{t^2 + \binom{m}{3} p^3 \left[ (1 - p^3) + 3 (m-3) p^2 (1-p) \right]}$.

Thus if we can find $l, m, p, s$ and $t$ such that $\binom{m}{2} p +
\binom{m}{3} p^3 + s + t \le \frac{m^3}{6(l-1)^2} - \frac{m}{6}$,
and
\begin{equation*}
\resizebox{1.0\hsize}{!}{$\left( \frac{m e (1 - p)^{\frac{l-1}{2}} }{l} \right)^l +\frac{\binom{m}{2} p ( 1 - p) }{s^2 + \binom{m}{2} p ( 1 - p) } + \frac{\binom{m}{3} p^3 \left[ (1 - p^3) + 3 (m-3) p^2 (1-p) \right]}{t^2 + \binom{m}{3} p^3 \left[ (1 - p^3) + 3 (m-3) p^2 (1-p) \right]} \left[ 1  -  \frac{\binom{m}{2} p ( 1 - p) }{s^2 + \binom{m}{2} p ( 1 - p) } \right] < 1$,}
\end{equation*}
then we prove that there is a suitable graph, and therefore $\overline{T_{n,l-1}}$ is not optimal for the $(3,l)$-problem.

A computer search determined that $l = 2074$, $m = 164397$, $p =
0.0051707$, $s = 14000$ and $t = 35000$ are suitable values.  Hence the graph with $2073$ disjoint cliques is not optimal
for the $(3,2074)$-problem.  Moreover, if $l > 2074$, then in
$\overline{T_{n,l-1}}$ we can replace $2073$ cliques by a graph with
fewer triangles.  Hence $\overline{T_{n,l-1}}$ is not optimal for the
$(3,l)$-problem for any $l \ge 2074$.

It would be interesting to find better constructions and to determine when
$\overline{T_{n,l-1}}$ stops being optimal for the $(3,l)$-problem.
Our flag algebra calculations suggest that it is still optimal for at least
the $(3,5)$- and $(3,6)$-problems.

\section{Flag algebra calculus}\label{flag_intro}

In this section we provide a brief introduction to the technique of
flag algebras.  First introduced by Razborov in \cite{raz_flag}, it has been applied
with great success to a wide variety of problems in extremal
combinatorics (see, for example, \cite{baber, hatami, hatami2, ravry, raz_hyper, raz_tri}).

We will begin with a general overview of the calculus, by
introducing some key definitions and providing some intuition behind
the machinery.  The second subsection will show how we express
extremal problems in the language of flag algebras.  In Appendix \ref{app_impflags} we discuss some practical considerations regarding
implementation of the method, to explain how we obtained our results in
the later sections.

\medskip

It is neither our goal to be rigorous nor thorough, but rather to
emphasize that the combinatorial arguments behind the flag algebra
calculus are as old as extremal combinatorics itself.  Indeed, the
main tools available to us are double-counting and the
Cauchy-Schwarz inequality.  To highlight this fact, we will use the
$(3,3)$-problem as a running example, and indeed, the proof
we obtain through flag algebras will be essentially the same as the
original proof Lorden gave in 1962.

The flag algebra calculus is powerful because it provides a
formalism through which the problem of finding relations between
subgraph densities can be reduced to a semi-definite programming
(SDP) problem.  This in turn enables the use of computers to find
solutions, with rigorous proofs, to problems in extremal
combinatorics.  For a more complete survey of the technique, we
refer you to the excellent expositions in \cite{keevash} and
\cite{ravry}, while for a technical specification of flag algebras,
we refer you to the original paper of Razborov \cite{raz_flag}.

\subsection{Basic definitions and notation} \label{sec31}

The flag algebra calculus is typically used to find the extremal
density of some fixed subgraph $J$ amongst graphs that avoid some
forbidden subgraph.  For our example, the $(3,3)$-problem, we wish
to minimize the density of triangles $K_3$ in graphs that do not
contain $\overline{K_3}$, the empty graph on $3$ vertices.  While
our definitions will be general, all our examples will come from
this setting.

\medskip

We say that a graph is \emph{admissible} if it contains no induced
copies of the forbidden graph. A \emph{type} $\sigma$ is an
admissible labeled graph on vertices $[k]$ for some non-negative
integer $k$ called the \emph{size} of $\sigma$, denoted by
$|\sigma|$. In what follows, an isomorphism between graphs must
preserve any labels that are present.

Given a type $\sigma$, a \emph{$\sigma$-flag} is an admissible graph
$F$ on a partially labeled vertex set, such that the subgraph
induced by the labeled vertices is isomorphic to $\sigma$. The
\emph{underlying graph} of the flag $F$ is the graph $F$ with all
labels removed. The \emph{size} of a flag is the number of vertices.
Note that when $\sigma$ is the \emph{trivial type} of size $0$
(denoted by $\sigma=0$), a $\sigma$-flag is just an usual unlabeled
admissible graph. We shall write $\mathcal{F}^\sigma_l$ for the
collection of all $\sigma$-flags of size $l$. Let
$\mathcal{F}^{\sigma}=\bigcup_{l \ge 0} \mathcal{F}^{\sigma}_l$.
When the type $\sigma$ is trivial, we shall omit the superscript
from our notation.

\medskip

Let us now define two fundamental concepts in our calculus, namely
those of flag densities in larger flags and graphs. Let $\sigma$ be
a type of size $k$, let $m\ge 1$ be an integer and let
\{$F_i\}_{i=1}^m$ be a collection of $\sigma$-flags of sizes $l_i =
|F_i| \ge k$. Given a $\sigma$-flag $F$ of order at least $l = k +
\sum_{i=1}^m (l_i-k)$, let $T \subseteq V(F)$ be the set of labeled
vertices of $F$. Now select disjoint subsets $X_i\subseteq V(F)
\setminus T$ of sizes $|X_i| = l_i-k$, uniformly at random. This is
possible because $F$ has at least $\sum_i (l_i - k)$ unlabeled
vertices. Denote by $E_i$ the event that the $\sigma$-flag induced
by $T \cup X_i$ is isomorphic to $F_i$, for $i\in [m]$. We define
$p_\sigma(F_1,F_2,\ldots,F_m;
F)\stackrel{def}{=}\mathbb{P}(\cap_{i=1}^m E_i)$ to be the
probability that all these events occur simultaneously.

If $G$ is just an admissible graph of order at least $l$, and not a
$\sigma$-flag, then there is no pre-labeled set of vertices $T$ that
induces the type $\sigma$.  Instead, we uniformly at random select a
partial labeling $L : [k] \rightarrow V(G)$. This random labeling
turns $G$ into a $\sigma'$-flag $F_L$, where the type $\sigma'$ is the labeled subgraph induced by the set of vertices $L([k])$. If $\sigma' = \sigma$, we can then
proceed as above, otherwise we say the events $E_i$ have probability
$0$. Finally, we average over all possible random labelings.
Formally, let $Y$ be the following random variable
\[
Y \stackrel{def}{=} \left\{\begin{array}{ll}
p_{\sigma}(F_1,F_2,\ldots, F_m; F_L) & \text{if } \sigma'=\sigma \\
0 & \text{otherwise}
\end{array}\right..
\]
Define $d_{\sigma}(F_1,\ldots, F_m;
G)\stackrel{def}{=}\mathbb{E}(Y)$ as the expected value of the
random variable $Y$. The quantities $p_\sigma(F_1,F_2,\ldots, F_m;
F)$ and $d_\sigma(F_1,F_2,\ldots, F_m;G)$ are called \emph{flag
densities} of $\{F_i\}_{i\in[m]}$ in $F$ and in $G$, respectively.
Clearly these flag densities are the same whenever $\sigma=0$, in
which case we omit the subscript from both notations.

\medskip

To better illustrate these definitions, we give some examples. Let
\emph{dot} be the only type of size one. Let $\rho$ and
$\overline{\rho}$ be the two \emph{dot}-flags of size two, and let
$Z_i$, for $1 \le i \le 5$, be the five admissible \emph{dot}-flags
of size three (recall that we are forbidding $\overline{K_3}$).
These flags are shown in Figure \ref{example_flags}.

\vspace{0.1in}
\begin{figure}[H]
\centering

\parbox{1in}{
\centering
\begin{tikzpicture}
  [scale=0.6,auto=left,every node/.style={circle, draw, fill=black!50,inner sep=0pt, minimum width=4pt}]
  \node (n1) at (180:1 cm) [label=left:$1$]{};

\end{tikzpicture}
\caption*{\emph{dot}} }
\parbox{1in}{
\centering
\begin{tikzpicture}
  [scale=0.6,auto=left,every node/.style={circle, draw, fill=black!50,inner sep=0pt, minimum width=4pt}]
  \node (n1) at (180:1 cm) [label=left:$1$]{};
  \node (n2) at (360:1 cm) {};

  \foreach \from/\to in {n1/n2}
    \draw (\from) -- (\to);

\end{tikzpicture}
\caption*{$\rho$} }
\parbox{1in}{
\centering
\begin{tikzpicture}
  [scale=0.6,auto=left,every node/.style={circle, draw, fill=black!50,inner sep=0pt, minimum width=4pt}]
  \node (n1) at (180:1 cm) [label=left:$1$]{};
  \node (n2) at (360:1 cm) {};

  \foreach \from/\to in {}
    \draw (\from) -- (\to);

\end{tikzpicture}
\caption*{$\overline{\rho}$} }

\vspace{0.3in}

\parbox{1in}{
\centering
\begin{tikzpicture}
  [scale=0.6,auto=left,every node/.style={circle, draw, fill=black!50,inner sep=0pt, minimum width=4pt}]
  \node (n1) at (180:1 cm) [label=left:$1$]{};
  \node (n2) at (300:1 cm) {};
  \node (n3) at (60:1 cm) {};

  \foreach \from/\to in {n1/n2}
    \draw (\from) -- (\to);

\end{tikzpicture}
\caption*{$Z_1$} }
\parbox{1in}{
\centering
\begin{tikzpicture}
  [scale=0.6,auto=left,every node/.style={circle, draw, fill=black!50,inner sep=0pt, minimum width=4pt}]
  \node (n1) at (180:1 cm) [label=left:$1$]{};
  \node (n2) at (300:1 cm) {};
  \node (n3) at (60:1 cm) {};

  \foreach \from/\to in {n2/n3}
    \draw (\from) -- (\to);

\end{tikzpicture}
\caption*{$Z_2$} }
\parbox{1in}{
\centering
\begin{tikzpicture}
  [scale=0.6,auto=left,every node/.style={circle, draw, fill=black!50,inner sep=0pt, minimum width=4pt}]
  \node (n1) at (180:1 cm) [label=left:$1$]{};
  \node (n2) at (300:1 cm) {};
  \node (n3) at (60:1 cm) {};

  \foreach \from/\to in {n1/n2,n1/n3}
    \draw (\from) -- (\to);

\end{tikzpicture}
\caption*{$Z_3$} }
\parbox{1in}{
\centering
\begin{tikzpicture}
  [scale=0.6,auto=left,every node/.style={circle, draw, fill=black!50,inner sep=0pt, minimum width=4pt}]
  \node (n1) at (180:1 cm) [label=left:$1$]{};
  \node (n2) at (300:1 cm) {};
  \node (n3) at (60:1 cm) {};

  \foreach \from/\to in {n1/n2,n2/n3}
    \draw (\from) -- (\to);

\end{tikzpicture}
\caption*{$Z_4$} }
\parbox{1in}{
\centering
\begin{tikzpicture}
  [scale=0.6,auto=left,every node/.style={circle, draw, fill=black!50,inner sep=0pt, minimum width=4pt}]
  \node (n1) at (180:1 cm) [label=left:$1$]{};
  \node (n2) at (300:1 cm) {};
  \node (n3) at (60:1 cm) {};

  \foreach \from/\to in {n1/n2,n1/n3,n2/n3}
    \draw (\from) -- (\to);

\end{tikzpicture}
\caption*{$Z_5$} }
\caption{Some examples of flags of type \emph{dot}.} \label{example_flags}
\end{figure}

We now compute the flag densities of $\rho$ and $\overline{\rho}$ in
the flags $Z_i$.  For example, to compute $p_{dot}( \overline{\rho};
Z_1)$, note that to induce a copy of $\overline{\rho}$ we must
choose an unlabeled non-neighbor of $1$.  As only one of the two
unlabeled vertices in $Z_1$ is a non-neighbor of $1$, we conclude
that $p_{dot}( \overline{\rho}; Z_1) = \frac{1}{2}$.  Similarly,
$p_{dot}( \rho; Z_3 ) = 1$, because to induce $\rho$ we must select
a neighbor of $1$, and all the unlabeled vertices in $Z_3$ are
neighbors of $1$.  The other flag densities are $p_{dot}(\rho; Z_5)
= p_{dot}(\overline{\rho};Z_2)=1$,
$p_{dot}(\rho;Z_1)=p_{dot}(\rho;Z_4)=p_{dot}(\overline{\rho};
Z_1)=p_{dot}(\overline{\rho};Z_4)=\frac12$, and $p_{dot}(\rho;Z_2) =
p_{dot}(\overline{\rho};Z_3) = p_{dot}(\overline{\rho};Z_5) = 0$.

\begin{figure}[H]
\centering
\begin{tikzpicture}
  [scale=0.6,auto=left,every node/.style={circle, draw, fill=black!50,inner sep=0pt, minimum width=4pt}]
  \node (n1) at (18:1 cm) {};
  \node (n2) at (90:1 cm)  {};
  \node (n3) at (162:1 cm)  {};
  \node (n4) at (234:1 cm)  {};
  \node (n5) at (306:1 cm)  {};

  \foreach \from/\to in {n1/n2,n1/n5,n2/n3,n2/n5,n3/n4,n3/n5,n4/n5}
    \draw (\from) -- (\to);

\end{tikzpicture}
\caption{Graph $W$.}
\label{example_graph}
\end{figure}

To see how to compute flag densities in an unlabeled graph, consider
$W$, the graph on $5$ vertices depicted in Figure
\ref{example_graph}. It is easy to see that $d_{dot}(\rho;W)$ and
$d_{dot}(\overline{\rho};W)$ are the edge and non-edge densities of
$W$ respectively, and so $d_{dot}(\rho;W)=\frac{7}{10}$ and
$d_{dot}(\overline{\rho};W)=\frac{3}{10}$.  The computation of
$d_{dot}(Z_i;W)$ is a little more involved.  As an example, we
explain how to compute $d_{dot}(Z_3;W)$.  Note that $Z_3$ consists
of two nonadjacent neighbors of the labeled vertex $1$.  Hence for
every vertex $v \in V(W)$, let $\kappa_v$ denote the number of
nonadjacent pairs neighbors of $v$ divided by the total number of
pairs of vertices in $V(W) \setminus \{ v \}$.  $d_{dot}(Z_3 ; W)$
is then the average of $\kappa_v$ over all vertices in $W$, which
comes out to $\frac{1}{6}$.  Computing the other flag densities
gives $d_{dot}(Z_1;W) = \frac{2}{15}$, $d_{dot}(Z_2;W) =
\frac{1}{15}$, $d_{dot}(Z_4;W) = \frac13$, and
$d_{dot}(Z_5,W)=\frac{3}{10}$.

We can also compute the joint flag densities of multiple flags. For
instance, let us consider $d_{dot}(\rho, \rho; W)$.  In this case,
we first randomly choose a vertex $v$ to be the labeled vertex.  We
must then make an \emph{ordered} choice of two vertices in $V(W)
\setminus \{v\}$, as we have two flags, each with one unlabeled
vertex.  If both of these vertices are neighbors of $v$, then we
have induced two copies of the flag $\rho$ (note that the adjacency
of these two vertices is unimportant).  Hence we obtain
$d_{dot}(\rho,\rho; W)$ by averaging over all vertices $v$ the ratio
of the number of ordered pairs of neighbors of $v$ to the number of
ordered pairs of vertices in $V(W) \setminus \{ v \}$.  In this
case, we have $d_{dot}( \rho, \rho ; W ) = \frac{7}{15}$.

\medskip

Suppose as before we have a type $\sigma$ of size $k$, a
$\sigma$-flag $F$ of size $l \ge k$, and an unlabeled graph $G$. To
compute $d_{\sigma}(F ; G)$, we averaged over all random partial
labelings of $G$ the probability of finding a flag isomorphic to
$F$. A simple double-counting argument shows that we can do the
averaging before the random labeling, which is the idea behind
Razborov's \emph{averaging operator}, as defined in Section 2.2 of
\cite{raz_flag}. Let $F|_0$ denote the unlabeled underlying graph of
$F$. We can compute $d_{\sigma}(F;G)$ by first computing
$d(F|_0;G)$, the probability that $l$ randomly chosen vertices in
$G$ form an induced copy of $F|_0$ as a subgraph.  Given this copy
of $F|_0$, we then randomly label $k$ of the $l$ vertices, and
compute the probability that these $k$ vertices are label-isomorphic
to $\sigma$.  This amounts to multiplying $d(F|_0; G)$ by a
\emph{normalizing factor} $q_{\sigma}(F)$, that is, $d_{\sigma}(F;G)
= q_{\sigma}(F) d(F|_0;G)= q_{\sigma}(F) p(F|_0;G)$.

We can interpret the normalizing factor as $q_{\sigma}(F) =
d_{\sigma}(F ; F|_0)$. From our previous example, we have
$q_{dot}(\rho)=q_{dot}(\overline{\rho})=q_{dot}(Z_5)=1$,
$q_{dot}(Z_3)=q_{dot}(Z_2)=\frac13$ and
$q_{dot}(Z_4)=q_{dot}(Z_1)=\frac23$. Since $q_{dot}(Z_5) = 1$, it
follows that $d_{dot}(Z_5;G)=d(K_3;G)$ is the triangle density of
$G$.

\medskip

There are more relations involving $d_\sigma$ and $p_\sigma$ than
the one mentioned previously. We will now state, without proof, a
basic fact about flag densities that can be proved easily by double
counting.
\begin{fact}[Chain rule]
\label{chain_rule}
If $\sigma$ is a type of size $k$, $m\ge 1$ is an integer, and $\{F_i\}_{i=1}^m$ is a family of $\sigma$-flags of sizes $|F_i|=l_i$, and $l \ge k + \sum_{i=1}^m (l_i - k)$ is an integer parameter, then
\begin{enumerate}
\item  For any $\sigma$-flag $F$ of order at least $l$, we have
\[
p_{\sigma}(F_1,\ldots,F_m;F) = \sum_{F' \in \mathcal{F}^{\sigma}_{l}} p_{\sigma}(F_1,\ldots, F_m;F') p_{\sigma}(F';F).
\]
\item  For any admissible graph $G$ of order at least $l$, we have
\[
d_{\sigma}(F_1,\ldots,F_m;G) = \sum_{H \in \mathcal{F}_{l}} d_{\sigma}(F_1,\ldots,F_m;H) d(H;G) =\sum_{F \in \mathcal{F}^{\sigma}_{l}} p_{\sigma}(F_1,\ldots, F_m;F) d_{\sigma}(F;G).
\]
\end{enumerate}
\end{fact}

If we apply the chain rule for $m=1$, we have the equation
$p_{\sigma}(F;F') = \sum_{F'' \in \mathcal{F}^{\sigma}_{l}}
p_{\sigma}(F;F'') p_{\sigma}(F'';F')$. For instance, this gives
\begin{align*}
    p_{dot}(\rho; F) &= p_{dot}(\rho; Z_1) p_{dot}(Z_1; F) +
p_{dot}(\rho; Z_2) p_{dot}(Z_2; F) + p_{dot}(\rho; Z_3) p_{dot}(Z_3; F) + \\
    & \qquad p_{dot}(\rho; Z_4) p_{dot}(Z_4; F) + p_{dot}(\rho; Z_5)
p_{dot}(Z_5; F) \\
    &= \frac{1}{2} p_{dot}(Z_1; F) + p_{dot}(Z_3; F) +
\frac{1}{2} p_{dot}(Z_4; F) + p_{dot}(Z_5; F).
\end{align*}

Similarly, we can expand $p_{dot}(\overline{\rho}; F) = \frac{1}{2}
p_{dot}(Z_1;F) + p_{dot}(Z_2; F) + \frac{1}{2} p_{dot}(Z_4; F)$.

For the ease of notation, we can express these two identities using
the syntax of flag algebras:
\begin{align}
\rho &= \frac{1}{2} Z_1 + Z_3 + \frac{1}{2} Z_4 + Z_5, \label{edgeexpansion} \\
\overline{\rho} &= \frac{1}{2} Z_1 + Z_2 + \frac{1}{2} Z_4. \notag
\end{align}
In this syntax, the equation $\sum_{i\in I} \alpha_i F_i = 0$ means
that for all sufficiently large $\sigma$-flags $F$, we have
$\sum_{i\in I} \alpha_i p_{\sigma}(F_i; F) = 0$, where $\alpha_i \in
\mathbb{R}$ for all $i\in I$. We call $\sum_{i\in I} \alpha_i F_i$
an \emph{eventually zero} expression.  We use $\mathcal{A}^{\sigma}$ to denote the set of linear combinations of flags of type $\sigma$. It is convenient to define a \emph{product} of flags in the following way:
\[
F_1\cdot F_2 \stackrel{def}{=} \sum_{F\in \mathcal{F}^{\sigma}_l} p_{\sigma}(F_1,F_2; F) F, \qquad F_1 \in \mathcal{F}^{\sigma}, F_2 \in \mathcal{F}^{\sigma}, l \ge |F_1| + |F_2| - |\sigma|.
\]
(Note that it does not matter what $l$ we
choose, as the difference will be an eventually zero expression.) For example, instead of
writing $p_{dot}(\rho, \rho; F) = p_{dot}(Z_3; F) + p_{dot}(Z_5;
F)$, we could simply write $\rho^2 = \rho \cdot \rho = Z_3 + Z_5$.
For the flags of our running example, involving $\overline{K_3}$-free graphs, the following equations are
also easily verifiable: $\rho^2 = Z_3 + Z_5$, $\overline{\rho}^2 =
Z_2$, and $\rho \cdot \overline{\rho} = \frac{1}{2} Z_4 +
\frac{1}{2} Z_1$. Combining these equations, we arrive at the
following equation, which we shall later require in Section
\ref{sec43}:
\begin{equation}
\label{stability_equation}
4 \rho^2 \cdot \overline{\rho}^2 = 4 Z_2 \cdot (Z_3 + Z_5) = (Z_4 + Z_1)^2.
\end{equation}

To further simplify the notation, we can extend the definitions of
$p_{\sigma}$ and $d_{\sigma}$ to
$\mathcal{A}^{\sigma}$ by making them linear in each
coordinate. For example, $p_{\sigma}(F_1 + 2 F_2, 4 F_3; F_4 -
F_5)=4p_{\sigma}(F_1,F_3;F_4)-
4p_{\sigma}(F_1,F_3;F_5)+8p_{\sigma}(F_2,F_3;F_4)-8p_{\sigma}(F_2,F_3;F_5)$.
The product notation simplifies these extended
definitions, because $p_{\sigma}(f_1\cdot f_2; f) =
p_{\sigma}(f_1,f_2;f)$ and $d_{\sigma}(f_1\cdot f_2; g) =
d_{\sigma}(f_1, f_2; g)$, for any $f_1,f_2,f\in
\mathcal{A}^{\sigma}$ and for any $g\in\mathcal{A}^0$.

\medskip

The last piece of notation we introduce is that of the averaging
operator.  Recall that for any $\sigma$-flag $F$, we had the
normalizing factors $q_{\sigma}(F)$ such that $d_{\sigma}(F; G) =
q_{\sigma}(F) p(F|_0 ; G)$. In the syntax of flag algebra, this
averaging operation is denoted by
$[[F]]_{\sigma}\stackrel{def}{=}q_{\sigma} F|_0$ .  We can extend
this linearly to all elements of $\mathcal{A}^{\sigma}$. For example
\[
[[\rho]]_{dot} = K_2, \quad [[Z_5]]_{dot} = K_3, \text{ and } \quad [[Z_4 + Z_2]]_{dot} = \frac23 P_2 + \frac13 \overline{P_2},
\]
where $P_2$ is a path of length two on three vertices, and
$\overline{P_2}$ is its complement. This notation is useful, because
$d_{\sigma}(f; g)=p([[f]]_{\sigma}; g)$ for any
$f\in\mathcal{A}^{\sigma}$ and for any $g\in\mathcal{A}^0$, and
hence we have a unified notation for both types of flag densities.

\subsection{Extremal problems in the flag algebra calculus}

Recall that the typical problem is to minimize the density of some
fixed graph $J$ amongst all admissible graphs $G$ not containing a
forbidden subgraph.  We will show how flag algebras can be applied
to this problem to reduce it to a semi-definite programming (SDP)
problem, which can then be solved numerically.

We may use the chain rule to obtain, for any $t \ge |J|$, the
equation $d(J;G) = \sum_{H \in \mathcal{F}_t} d(J;H) d(H;G)$. Since
$\sum_{H\in \mathcal{F}_t} d(H;G)=1$, we have
\[
d(J;G) \ge \min_{H\in \mathcal{F}_t} d(J;H),
\]
which is a bound that clearly does not depend on $G$.

This inequality is often very weak, since it only uses very local
considerations about the subgraphs $H \in \mathcal{F}_t$, and does
not take into account how the subgraphs fit together in the larger
graph $G$; that is, how they intersect. For instance, returning to
our example of the $(3,3)$-problem, where $J=K_3$ and $t=3$, we
obtain $d(K_3;G) \ge \min_{H\in \mathcal{F}_3} d(K_3; H) =
d(K_3;P_2)=0$, which is the most trivial bound. However, by
considering how the graphs in $\mathcal{F}_3$ must intersect in $G$,
one might hope to find inequalities of the form $\sum_{H \in
\mathcal{F}_t} \alpha_H d(H;G) \ge 0$, such that when we combine
them with the initial identity, we get
\[
d(J;G) \ge d(J;G) - \sum_{H \in \mathcal{F}_t} \alpha_H d(H;G) = \sum_{H \in \mathcal{F}_t} (d(J;H) - \alpha_H) d(H;G) \ge \min_{H\in \mathcal{F}_t} \{d(J;H) - \alpha_H\}.
\]

Since $\alpha_H$ can be negative for some graphs $H$, the hope is
that this will improve the low coefficients by transferring weight
from high coefficients. In order to find such inequalities, we need
another property of the flag densities.

\begin{fact}
\label{small_error} If $\sigma$ is a type of size $k$, $m\ge 1$ is
an integer, $\{F_i\}_{i=1}^m$ is a family of $\sigma$-flags of sizes
$|F_i|=l_i$, and $l \ge k + \sum_{i=1}^m (l_i - k)$ is an integer,
then for any flag $F$ of order $n \ge l$, we have
\[
p_{\sigma}(F_1,\ldots,F_m;F) = \left[\prod_{i=1}^m p_{\sigma}(F_i;F)\right] + O(1/n).
\]
\end{fact}

One can prove Fact \ref{small_error} by noting that, if we drop the
requirement that the sets $X_i$ are disjoint in the definition of
$p_{\sigma}(F_1,\ldots,F_m;F)$, the events $E_i$ will become
independent, and thus $\mathbb{P}(\cap_{i=1}^m E_i) = \prod_{i=1}^m
\mathbb{P}(E_i) = \prod_{i=1}^m p_{\sigma}(F_i;F)$.  The error
introduced is the probability that these sets $X_i$ will intersect
in $F$, which is $O(1 / n)$.  It is tempting to claim a similar product formula for the unlabeled flag densities $d_{\sigma}$, but we cannot do so.  In the above equation, it is essential that all the $\sigma$-flags $F_i$ share the same labeled type $\sigma$, and hence we require $F$ to be a $\sigma$-flag.

\medskip

We are now ready to establish some inequalities. Let's first fix a
type $\sigma$ of size $k$. If $Q$ is any positive semi-definite
$|\mathcal{F}_l^{\sigma}|\times|\mathcal{F}_l^{\sigma}|$ matrix with
rows and columns indexed by the same set $\mathcal{F}_l^{\sigma}$,
where $l\ge k$, define
\[
Q\{\mathcal{F}_l^{\sigma}\} \stackrel{def}{=} \sum_{F_1,F_2 \in \mathcal{F}_l^{\sigma}} Q_{F_1,F_2} F_1 \cdot F_2 \in \mathcal{A}^{\sigma}.
\]
Since $Q$ was chosen to be positive semi-definite, we have
\[ p_{\sigma}(Q\{\mathcal{F}_l^{\sigma}\}; F)=\sum_{F_1,F_2 \in
\mathcal{F}_l^{\sigma}} Q_{F_1,F_2} p_{\sigma}(F_1;F)
p_{\sigma}(F_2;F) \ge 0 \]
for any $\sigma$-flags $F$ of order at
least $t=2l - k$. When averaging, we do not necessarily have
$p([[Q\{\mathcal{F}_l^{\sigma}\}]]_\sigma; G) \ge 0$ for an
admissible graph $G$ of order $n\ge t$, but we do have the following
inequality:
\begin{align*}
[[Q]]_{\sigma}(G) &\stackrel{def}{=} p([[Q\{\mathcal{F}_l^{\sigma}\}]]_{\sigma};G) = \sum_{F_1,F_2 \in \mathcal{F}_l^{\sigma}} Q_{F_1,F_2} d_{\sigma}(F_1,F_2;G) \\
&= \sum_{F_1,F_2 \in \mathcal{F}_l^{\sigma}} Q_{F_1,F_2} \left(
\sum_{F\in \mathcal{F}_{n}^{\sigma}} p_{\sigma}(F_1,F_2;F)
d_{\sigma}(F;G) \right) \\
&= \sum_{F\in \mathcal{F}_{n}^{\sigma}}
\left(\sum_{F_1,F_2 \in \mathcal{F}_l^{\sigma}} Q_{F_1,F_2}
p_{\sigma}(F_1,F_2;F)\right) d_{\sigma}(F;G) \\
&= \sum_{F\in \mathcal{F}_{n}^{\sigma}} \left(\sum_{F_1,F_2 \in
\mathcal{F}_l^{\sigma}} Q_{F_1,F_2} p_{\sigma}(F_1;F)
p_{\sigma}(F_2;F)\right) d_{\sigma}(F;G) + O(1/n) \ge
o_{n\to\infty}(1).
\end{align*}

 Therefore, when $n$ is large, we have that
$[[Q]]_{\sigma}(G)$ is asymptotically non-negative. For each
admissible graph $H$ of size exactly $t$, let $\alpha_H =
[[Q]]_{\sigma}(H)= \sum_{F_1,F_2 \in \mathcal{F}_t^{\sigma}}
Q_{F_1,F_2} d_{\sigma}(F_1,F_2;H)$. We then have
\[
 [[Q]]_{\sigma}(G) = \sum_{H\in \mathcal{F}_t} \alpha_H d(H;G) \ge o_{n\to\infty}(1).
\]
The expression in the middle of the above equation is called the
\emph{expansion} of $[[Q]]_{\sigma}(G)$ in graphs of size $t$, with
$\alpha_H$ the coefficients of the expansion. For the sake of
conciseness, we often omit the parameter $G$ and express this
asymptotic inequality (combined with the expansion in size $t$) in
the syntax of flag algebras
\begin{equation}
\label{positivity_flags}
 [[Q]]_{\sigma} \stackrel{def}{=} [[Q\{\mathcal{F}_l^{\sigma}\}]]_{\sigma}= \bigg[\bigg[\sum_{F_1,F_2\in\mathcal{F}^{\sigma}_l} Q_{F_1,F_2} F_1 \cdot F_2\bigg]\bigg]_{\sigma}=\sum_{H\in \mathcal{F}_t} \alpha_H H \ge 0.
\end{equation}
(Note that all inequalities between flags stated in the language of
flag algebras are asymptotic.)

For a concrete example, we return to the $(3,3)$-problem. If we use
the type $\sigma=dot$, flags of size $l=2$, expand in graphs of size
$t=3$, and consider
\[
Q = \begin{pmatrix}
+\frac34 & -\frac34 \\
-\frac34 & +\frac34
\end{pmatrix},
\]
where the rows and columns are indexed by $\rho$ and
$\overline{\rho}$ (in that order), we obtain
$Q\{\mathcal{F}_2^{dot}\}=\frac34 (\rho - \overline{\rho})^2=\frac34
(-Z_1-Z_4+Z_2+Z_3+Z_5)$. This expansion is obtained by substituting
the expressions for $\rho^2$, $\overline{\rho}^2$ and $\rho \cdot
\overline{\rho}$ that are given above Equation
\ref{stability_equation}. Averaging gives $[[Q]]_{\sigma} = \frac34
[[(\rho - \overline{\rho})^2]]_{dot} = \frac34 K_3 - \frac14 P_2 -
\frac14 \overline{P_2}$.  Recall that $K_3 + P_2 + \overline{P_2} =
1$, since we are only considering $\overline{K_3}$-free graphs.
Therefore $d(K_3;G) \ge \min_{H\in \mathcal{F}_3} \left\{ d(K_3;H) -
[[Q]]_{\sigma}(H) \right\} = \frac14$, which is the correct bound
for the $(3,3)$-problem.

\medskip

In general, if we have more than one inequality available, we can
combine them together, provided they are all expanded in the same
size $t$. Suppose we have $r$ inequalities given by the positive
semi-definite matrices $Q_i$ of the $\sigma_i$-flags of size $l_i$.
Adding them together, we obtain
\[
\sum_{i=1}^r [[Q_i]]_{\sigma_i} = \sum_{H \in \mathcal{F}_t} \alpha_H H \ge  0,
\]
where
\[
\alpha_H = \sum_{i=1}^r \left(\sum_{F_1,F_2 \in \mathcal{F}_{l_i}^{\sigma_i}} (Q_i)_{F_1,F_2} d_{\sigma_i}(F_1,F_2;H)\right),
\]
and we want to maximize $\min_{H \in \mathcal{F}_t} \left\{d(J;H) -
\alpha_H \right\}$.

\medskip

Thus we have transformed the original problem of finding a maximum
lower bound for $d(J;G)$ into a linear system involving the variables $(Q_i)_{F_k,F_l}$.  As we have the constraint that the matrices $Q_i$ should be positive semi-definite, this is a semi-definite programming problem.  To take the
minimum coefficient in the expansion, we introduce an artificial
variable $y$, and require it to be bounded above by all the
coefficients.  Hence we have the following SDP problem in the variables $y$ and $( Q_i )_{F_1,F_2}$:

\medskip

Maximize $y$, subject to the constraints:
\begin{itemize}
\item $s_H = d(J;H) - \sum_{i=1}^r \left(\sum_{F_1,F_2 \in \mathcal{F}_{l_i}^{\sigma_i}} (Q_i)_{F_1,F_2} d_{\sigma_i}(F_1,F_2;H)\right) - y \ge 0$ for all $H\in \mathcal{F}_t$. (The variables $s_H$ are called \emph{surplus}
variables.)
\item $Q_i$ is positive semi-definite for $i \in [r]$. (The matrices $Q_i$ are often called the \emph{block variables} of the SDP problem. We can assume without loss of generality that each $Q_i$ is symmetric, as otherwise we could replace $Q_i$ by $(Q_i + Q_i^T)
/2$.)
\end{itemize}

\medskip

A computer can solve this SDP problem numerically, allowing for an
efficient determination of the inequalities required to prove the
extremal problem.  For some practical remarks on the implementation
of flag algebras, please see Appendix \ref{app_impflags}.  We note
at this point, as shall be seen in Section \ref{sec43}, that the
solution to the SDP problem need not only give the asymptotic bound,
but can also provide some structural information about the extremal
graphs.

\section{The $(4,3)$-problem} \label{sec43}

In this section we will apply the flag algebra calculus to solve the
$(4,3)$-problem.  Recall in the $(4,3)$-problem we are interested in
finding the minimum number of $4$-cliques in a graph with
independence number less than $3$.  We prove that any graph on $n$
vertices with independence number at most $2$ must contain at least
$\frac{3}{25} \binom{n}{4} + O(n^3)$ $4$-cliques.  This bound is
attained by a balanced blow-up of $C_5$, which Nikiforov conjectured
to be optimal in \cite{nikiforov_pre}.

The first subsection contains our flag algebra results, which leads
to the asymptotic minimum density of $4$-cliques.  In the second
subsection we use the structural information from the flag algebras
to derive a stability result.  This allows us to determine the value
of $f(n,4,3)$ exactly for large $n$, and we show that a
nearly-balanced blow-up of $C_5$ is the unique extremal graph.

\subsection{The asymptotic result} \label{43_flags}

We begin by listing the admissible graphs of size $5$, the types
used in the proof, and the corresponding flags.  Note that the flags
of size 3 and type \emph{dot} in Figure \ref{dot_flags} are those we
used as examples in Section \ref{sec31}, Figure \ref{example_flags}.

\begin{figure}[H]
\centering
\parbox{0.8in}{
\centering
\begin{tikzpicture}
  [scale=0.6,auto=left,every node/.style={circle, draw, fill=black!50,inner sep=0pt, minimum width=4pt}]
  \node (n1) at (18:1 cm) {};
  \node (n2) at (90:1 cm)  {};
  \node (n3) at (162:1 cm)  {};
  \node (n4) at (234:1 cm)  {};
  \node (n5) at (306:1 cm)  {};

  \foreach \from/\to in {n2/n3,n2/n4,n2/n5,n3/n4,n3/n5,n4/n5}
    \draw (\from) -- (\to);

\end{tikzpicture}
\caption*{$G_1$} }
\parbox{0.8in}{
\centering
\begin{tikzpicture}
  [scale=0.6,auto=left,every node/.style={circle, draw, fill=black!50,inner sep=0pt, minimum width=4pt}]
  \node (n1) at (18:1 cm) {};
  \node (n2) at (90:1 cm)  {};
  \node (n3) at (162:1 cm)  {};
  \node (n4) at (234:1 cm)  {};
  \node (n5) at (306:1 cm)  {};

  \foreach \from/\to in {n1/n5,n2/n3,n2/n4,n3/n4,n3/n5,n4/n5}
    \draw (\from) -- (\to);

\end{tikzpicture}
\caption*{$G_2$} }
\parbox{0.8in}{
\centering
\begin{tikzpicture}
  [scale=0.6,auto=left,every node/.style={circle, draw, fill=black!50,inner sep=0pt, minimum width=4pt}]
  \node (n1) at (18:1 cm) {};
  \node (n2) at (90:1 cm)  {};
  \node (n3) at (162:1 cm)  {};
  \node (n4) at (234:1 cm)  {};
  \node (n5) at (306:1 cm)  {};

  \foreach \from/\to in {n1/n5,n2/n3,n2/n4,n2/n5,n3/n4,n3/n5,n4/n5}
    \draw (\from) -- (\to);

\end{tikzpicture}
\caption*{$G_3$} }
\parbox{0.8in}{
\centering
\begin{tikzpicture}
  [scale=0.6,auto=left,every node/.style={circle, draw, fill=black!50,inner sep=0pt, minimum width=4pt}]
  \node (n1) at (18:1 cm) {};
  \node (n2) at (90:1 cm)  {};
  \node (n3) at (162:1 cm)  {};
  \node (n4) at (234:1 cm)  {};
  \node (n5) at (306:1 cm)  {};

  \foreach \from/\to in {n1/n4,n1/n5,n2/n3,n2/n4,n2/n5,n3/n4,n3/n5,n4/n5}
    \draw (\from) -- (\to);

\end{tikzpicture}
\caption*{$G_4$} }
\parbox{0.8in}{
\centering
\begin{tikzpicture}
  [scale=0.6,auto=left,every node/.style={circle, draw, fill=black!50,inner sep=0pt, minimum width=4pt}]
  \node (n1) at (18:1 cm) {};
  \node (n2) at (90:1 cm)  {};
  \node (n3) at (162:1 cm)  {};
  \node (n4) at (234:1 cm)  {};
  \node (n5) at (306:1 cm)  {};

  \foreach \from/\to in {n1/n3,n1/n5,n2/n4,n2/n5,n3/n4,n3/n5,n4/n5}
    \draw (\from) -- (\to);

\end{tikzpicture}
\caption*{$G_5$} }
\parbox{0.8in}{
\centering
\begin{tikzpicture}
  [scale=0.6,auto=left,every node/.style={circle, draw, fill=black!50,inner sep=0pt, minimum width=4pt}]
  \node (n1) at (18:1 cm) {};
  \node (n2) at (90:1 cm)  {};
  \node (n3) at (162:1 cm)  {};
  \node (n4) at (234:1 cm)  {};
  \node (n5) at (306:1 cm)  {};

  \foreach \from/\to in {n1/n3,n1/n4,n1/n5,n2/n3,n2/n4,n2/n5,n3/n4,n3/n5,n4/n5}
    \draw (\from) -- (\to);

\end{tikzpicture}
\caption*{$G_6$} }
\parbox{0.8in}{
\centering
\begin{tikzpicture}
  [scale=0.6,auto=left,every node/.style={circle, draw, fill=black!50,inner sep=0pt, minimum width=4pt}]
  \node (n1) at (18:1 cm) {};
  \node (n2) at (90:1 cm)  {};
  \node (n3) at (162:1 cm)  {};
  \node (n4) at (234:1 cm)  {};
  \node (n5) at (306:1 cm)  {};

  \foreach \from/\to in {n1/n2,n3/n4,n3/n5,n4/n5}
    \draw (\from) -- (\to);

\end{tikzpicture}
\caption*{$G_7$} }
\parbox{0.8in}{
\centering
\begin{tikzpicture}
  [scale=0.6,auto=left,every node/.style={circle, draw, fill=black!50,inner sep=0pt, minimum width=4pt}]
  \node (n1) at (18:1 cm) {};
  \node (n2) at (90:1 cm)  {};
  \node (n3) at (162:1 cm)  {};
  \node (n4) at (234:1 cm)  {};
  \node (n5) at (306:1 cm)  {};

  \foreach \from/\to in {n1/n2,n2/n5,n3/n4,n3/n5,n4/n5}
    \draw (\from) -- (\to);

\end{tikzpicture}
\caption*{$G_8$} }
\parbox{0.8in}{
\centering
\begin{tikzpicture}
  [scale=0.6,auto=left,every node/.style={circle, draw, fill=black!50,inner sep=0pt, minimum width=4pt}]
  \node (n1) at (18:1 cm) {};
  \node (n2) at (90:1 cm)  {};
  \node (n3) at (162:1 cm)  {};
  \node (n4) at (234:1 cm)  {};
  \node (n5) at (306:1 cm)  {};

  \foreach \from/\to in {n1/n2,n1/n5,n2/n5,n3/n4,n3/n5,n4/n5}
    \draw (\from) -- (\to);

\end{tikzpicture}
\caption*{$G_9$} }
\parbox{0.8in}{
\centering
\begin{tikzpicture}
  [scale=0.6,auto=left,every node/.style={circle, draw, fill=black!50,inner sep=0pt, minimum width=4pt}]
  \node (n1) at (18:1 cm) {};
  \node (n2) at (90:1 cm)  {};
  \node (n3) at (162:1 cm)  {};
  \node (n4) at (234:1 cm)  {};
  \node (n5) at (306:1 cm)  {};

  \foreach \from/\to in {n1/n2,n1/n4,n2/n5,n3/n4,n3/n5,n4/n5}
    \draw (\from) -- (\to);

\end{tikzpicture}
\caption*{$G_{10}$} }
\parbox{0.8in}{
\centering
\begin{tikzpicture}
  [scale=0.6,auto=left,every node/.style={circle, draw, fill=black!50,inner sep=0pt, minimum width=4pt}]
  \node (n1) at (18:1 cm) {};
  \node (n2) at (90:1 cm)  {};
  \node (n3) at (162:1 cm)  {};
  \node (n4) at (234:1 cm)  {};
  \node (n5) at (306:1 cm)  {};

  \foreach \from/\to in {n1/n2,n1/n3,n2/n4,n3/n5,n4/n5}
    \draw (\from) -- (\to);

\end{tikzpicture}
\caption*{$G_{11}$} }
\parbox{0.8in}{
\centering
\begin{tikzpicture}
  [scale=0.6,auto=left,every node/.style={circle, draw, fill=black!50,inner sep=0pt, minimum width=4pt}]
  \node (n1) at (18:1 cm) {};
  \node (n2) at (90:1 cm)  {};
  \node (n3) at (162:1 cm)  {};
  \node (n4) at (234:1 cm)  {};
  \node (n5) at (306:1 cm)  {};

  \foreach \from/\to in {n1/n2,n1/n3,n2/n4,n2/n5,n3/n4,n3/n5,n4/n5}
    \draw (\from) -- (\to);

\end{tikzpicture}
\caption*{$G_{12}$} }
\parbox{0.8in}{
\centering
\begin{tikzpicture}
  [scale=0.6,auto=left,every node/.style={circle, draw, fill=black!50,inner sep=0pt, minimum width=4pt}]
  \node (n1) at (18:1 cm) {};
  \node (n2) at (90:1 cm)  {};
  \node (n3) at (162:1 cm)  {};
  \node (n4) at (234:1 cm)  {};
  \node (n5) at (306:1 cm)  {};

  \foreach \from/\to in {n1/n2,n1/n3,n1/n5,n2/n4,n2/n5,n3/n4,n3/n5,n4/n5}
    \draw (\from) -- (\to);

\end{tikzpicture}
\caption*{$G_{13}$} }
\parbox{0.8in}{
\centering
\begin{tikzpicture}
  [scale=0.6,auto=left,every node/.style={circle, draw, fill=black!50,inner sep=0pt, minimum width=4pt}]
  \node (n1) at (18:1 cm) {};
  \node (n2) at (90:1 cm)  {};
  \node (n3) at (162:1 cm)  {};
  \node (n4) at (234:1 cm)  {};
  \node (n5) at (306:1 cm)  {};

  \foreach \from/\to in {n1/n2,n1/n3,n1/n4,n1/n5,n2/n3,n2/n4,n2/n5,n3/n4,n3/n5,n4/n5}
    \draw (\from) -- (\to);

\end{tikzpicture}
\caption*{$G_{14}$} } \caption{Graphs of size $5$ with independence
number at most $2$.}
\end{figure}

\begin{figure}[H]
\centering
\parbox{0.8in}{
\centering
\begin{tikzpicture}
  [scale=0.6,auto=left,every node/.style={circle, draw, fill=black!50,inner sep=0pt, minimum width=4pt}]
  \node (n3) at (60:1 cm) [label=right:$3$]{};
  \node (n1) at (180:1 cm) [label=left:$1$]{};
  \node (n2) at (300:1 cm) [label=right:$2$]{};

  \foreach \from/\to in {n1/n2,n1/n3}
    \draw (\from) -- (\to);

\end{tikzpicture}
\caption*{$\tau_1$} }
\parbox{0.8in}{
\centering
\begin{tikzpicture}
  [scale=0.6,auto=left,every node/.style={circle, draw, fill=black!50,inner sep=0pt, minimum width=4pt}]
  \node (n4) at (45:1 cm) {};
  \node (n2) at (135:1 cm) [label=left:$2$]{};
  \node (n1) at (225:1 cm) [label=left:$1$]{};
  \node (n3) at (315:1 cm) [label=right:$3$]{};

  \foreach \from/\to in {n1/n2,n1/n3,n1/n4,n3/n4}
    \draw (\from) -- (\to);

\end{tikzpicture}
\caption*{$M_1$} }
\parbox{0.8in}{
\centering
\begin{tikzpicture}
  [scale=0.6,auto=left,every node/.style={circle, draw, fill=black!50,inner sep=0pt, minimum width=4pt}]
  \node (n4) at (45:1 cm) {};
  \node (n2) at (135:1 cm) [label=left:$2$]{};
  \node (n1) at (225:1 cm) [label=left:$1$]{};
  \node (n3) at (315:1 cm) [label=right:$3$]{};

  \foreach \from/\to in {n1/n2,n1/n3,n1/n4,n2/n4}
    \draw (\from) -- (\to);

\end{tikzpicture}
\caption*{$M_2$} }
\parbox{0.8in}{
\centering
\begin{tikzpicture}
  [scale=0.6,auto=left,every node/.style={circle, draw, fill=black!50,inner sep=0pt, minimum width=4pt}]
  \node (n4) at (45:1 cm) {};
  \node (n2) at (135:1 cm) [label=left:$2$]{};
  \node (n1) at (225:1 cm) [label=left:$1$]{};
  \node (n3) at (315:1 cm) [label=right:$3$]{};

  \foreach \from/\to in {n1/n2,n1/n3,n3/n4}
    \draw (\from) -- (\to);

\end{tikzpicture}
\caption*{$M_3$} }
\parbox{0.8in}{
\centering
\begin{tikzpicture}
  [scale=0.6,auto=left,every node/.style={circle, draw, fill=black!50,inner sep=0pt, minimum width=4pt}]
  \node (n4) at (45:1 cm) {};
  \node (n2) at (135:1 cm) [label=left:$2$]{};
  \node (n1) at (225:1 cm) [label=left:$1$]{};
  \node (n3) at (315:1 cm) [label=right:$3$]{};

  \foreach \from/\to in {n1/n2,n1/n3,n2/n4}
    \draw (\from) -- (\to);

\end{tikzpicture}
\caption*{$M_4$} } \caption{Type $\tau_1$ and its flags of size $4$.}
\end{figure}

\begin{figure}[H]
\centering
\parbox{0.8in}{
\centering
\begin{tikzpicture}
  [scale=0.6,auto=left,every node/.style={circle, draw, fill=black!50,inner sep=0pt, minimum width=4pt}]
  \node (n3) at (60:1 cm) [label=right:$3$]{};
  \node (n1) at (180:1 cm) [label=left:$1$]{};
  \node (n2) at (300:1 cm) [label=right:$2$]{};

  \foreach \from/\to in {n1/n2,n1/n3,n2/n3}
    \draw (\from) -- (\to);

\end{tikzpicture}
\caption*{$\tau_2$} }
\parbox{0.8in}{
\centering
\begin{tikzpicture}
  [scale=0.6,auto=left,every node/.style={circle, draw, fill=black!50,inner sep=0pt, minimum width=4pt}]
  \node (n4) at (45:1 cm) {};
  \node (n2) at (135:1 cm) [label=left:$2$]{};
  \node (n1) at (225:1 cm) [label=left:$1$]{};
  \node (n3) at (315:1 cm) [label=right:$3$]{};

  \foreach \from/\to in {n1/n2,n1/n3,n2/n3}
    \draw (\from) -- (\to);

\end{tikzpicture}
\caption*{$N_1$} }
\parbox{0.8in}{
\centering
\begin{tikzpicture}
  [scale=0.6,auto=left,every node/.style={circle, draw, fill=black!50,inner sep=0pt, minimum width=4pt}]
  \node (n4) at (45:1 cm) {};
  \node (n2) at (135:1 cm) [label=left:$2$]{};
  \node (n1) at (225:1 cm) [label=left:$1$]{};
  \node (n3) at (315:1 cm) [label=right:$3$]{};

  \foreach \from/\to in {n1/n2,n1/n3,n2/n3,n3/n4}
    \draw (\from) -- (\to);

\end{tikzpicture}
\caption*{$N_2$} }
\parbox{0.8in}{
\centering
\begin{tikzpicture}
  [scale=0.6,auto=left,every node/.style={circle, draw, fill=black!50,inner sep=0pt, minimum width=4pt}]
  \node (n4) at (45:1 cm) {};
  \node (n2) at (135:1 cm) [label=left:$2$]{};
  \node (n1) at (225:1 cm) [label=left:$1$]{};
  \node (n3) at (315:1 cm) [label=right:$3$]{};

  \foreach \from/\to in {n1/n2,n1/n3,n2/n3,n2/n4}
    \draw (\from) -- (\to);

\end{tikzpicture}
\caption*{$N_3$} }
\parbox{0.8in}{
\centering
\begin{tikzpicture}
  [scale=0.6,auto=left,every node/.style={circle, draw, fill=black!50,inner sep=0pt, minimum width=4pt}]
  \node (n4) at (45:1 cm) {};
  \node (n2) at (135:1 cm) [label=left:$2$]{};
  \node (n1) at (225:1 cm) [label=left:$1$]{};
  \node (n3) at (315:1 cm) [label=right:$3$]{};

  \foreach \from/\to in {n1/n2,n1/n3,n2/n3,n1/n4}
    \draw (\from) -- (\to);

\end{tikzpicture}
\caption*{$N_4$} }
\parbox{0.8in}{
\centering
\begin{tikzpicture}
  [scale=0.6,auto=left,every node/.style={circle, draw, fill=black!50,inner sep=0pt, minimum width=4pt}]
  \node (n4) at (45:1 cm) {};
  \node (n2) at (135:1 cm) [label=left:$2$]{};
  \node (n1) at (225:1 cm) [label=left:$1$]{};
  \node (n3) at (315:1 cm) [label=right:$3$]{};

  \foreach \from/\to in {n1/n2,n1/n3,n2/n3,n2/n4,n3/n4}
    \draw (\from) -- (\to);

\end{tikzpicture}
\caption*{$N_5$} }
\parbox{0.8in}{
\centering
\begin{tikzpicture}
  [scale=0.6,auto=left,every node/.style={circle, draw, fill=black!50,inner sep=0pt, minimum width=4pt}]
  \node (n4) at (45:1 cm) {};
  \node (n2) at (135:1 cm) [label=left:$2$]{};
  \node (n1) at (225:1 cm) [label=left:$1$]{};
  \node (n3) at (315:1 cm) [label=right:$3$]{};

  \foreach \from/\to in {n1/n2,n1/n3,n2/n3,n1/n4,n3/n4}
    \draw (\from) -- (\to);

\end{tikzpicture}
\caption*{$N_6$} }
\parbox{0.8in}{
\centering
\begin{tikzpicture}
  [scale=0.6,auto=left,every node/.style={circle, draw, fill=black!50,inner sep=0pt, minimum width=4pt}]
  \node (n4) at (45:1 cm) {};
  \node (n2) at (135:1 cm) [label=left:$2$]{};
  \node (n1) at (225:1 cm) [label=left:$1$]{};
  \node (n3) at (315:1 cm) [label=right:$3$]{};

  \foreach \from/\to in {n1/n2,n1/n3,n2/n3,n1/n4,n2/n4}
    \draw (\from) -- (\to);

\end{tikzpicture}
\caption*{$N_7$} }
\parbox{0.8in}{
\centering
\begin{tikzpicture}
  [scale=0.6,auto=left,every node/.style={circle, draw, fill=black!50,inner sep=0pt, minimum width=4pt}]
  \node (n4) at (45:1 cm) {};
  \node (n2) at (135:1 cm) [label=left:$2$]{};
  \node (n1) at (225:1 cm) [label=left:$1$]{};
  \node (n3) at (315:1 cm) [label=right:$3$]{};

  \foreach \from/\to in {n1/n2,n1/n3,n2/n3,n1/n4,n2/n4,n3/n4}
    \draw (\from) -- (\to);

\end{tikzpicture}
\caption*{$N_8$} } \caption{Type $\tau_2$ and its flags of size $4$.}
\end{figure}

\begin{figure}[H]
\centering
\parbox{0.8in}{
\centering
\begin{tikzpicture}
  [scale=0.6,auto=left,every node/.style={circle, draw, fill=black!50,inner sep=0pt, minimum width=4pt}]
  \node (n1) at (0:0 cm) [label=left:$1$]{};

  \foreach \from/\to in {}
    \draw (\from) -- (\to);

\end{tikzpicture}
\caption*{$dot$} }
\parbox{0.8in}{
\centering
\begin{tikzpicture}
  [scale=0.6,auto=left,every node/.style={circle, draw, fill=black!50,inner sep=0pt, minimum width=4pt}]
  \node (n2) at (60:1 cm) {};
  \node (n1) at (180:1 cm) [label=left:$1$]{};
  \node (n3) at (300:1 cm) {};

  \foreach \from/\to in {n1/n3}
    \draw (\from) -- (\to);

\end{tikzpicture}
\caption*{$Z_1$} }
\parbox{0.8in}{
\centering
\begin{tikzpicture}
  [scale=0.6,auto=left,every node/.style={circle, draw, fill=black!50,inner sep=0pt, minimum width=4pt}]
  \node (n2) at (60:1 cm) {};
  \node (n1) at (180:1 cm) [label=left:$1$]{};
  \node (n3) at (300:1 cm) {};

  \foreach \from/\to in {n2/n3}
    \draw (\from) -- (\to);

\end{tikzpicture}
\caption*{$Z_2$} }
\parbox{0.8in}{
\centering
\begin{tikzpicture}
  [scale=0.6,auto=left,every node/.style={circle, draw, fill=black!50,inner sep=0pt, minimum width=4pt}]
  \node (n2) at (60:1 cm) {};
  \node (n1) at (180:1 cm) [label=left:$1$]{};
  \node (n3) at (300:1 cm) {};

  \foreach \from/\to in {n1/n2,n1/n3}
    \draw (\from) -- (\to);

\end{tikzpicture}
\caption*{$Z_3$} }
\parbox{0.8in}{
\centering
\begin{tikzpicture}
  [scale=0.6,auto=left,every node/.style={circle, draw, fill=black!50,inner sep=0pt, minimum width=4pt}]
  \node (n2) at (60:1 cm) {};
  \node (n1) at (180:1 cm) [label=left:$1$]{};
  \node (n3) at (300:1 cm) {};

  \foreach \from/\to in {n1/n3,n2/n3}
    \draw (\from) -- (\to);

\end{tikzpicture}
\caption*{$Z_4$} }
\parbox{0.8in}{
\centering
\begin{tikzpicture}
  [scale=0.6,auto=left,every node/.style={circle, draw, fill=black!50,inner sep=0pt, minimum width=4pt}]
  \node (n2) at (60:1 cm) {};
  \node (n1) at (180:1 cm) [label=left:$1$]{};
  \node (n3) at (300:1 cm) {};

  \foreach \from/\to in {n1/n2,n2/n3,n1/n3}
    \draw (\from) -- (\to);

\end{tikzpicture}
\caption*{$Z_5$} } \caption{Type \emph{dot} and its flags of size $3$.}
\label{dot_flags}
\end{figure}

For each of the types used in the proof, we express the corresponding positive semi-definite matrices as a sum of squares.  In the lemmas that follow, we give these sums of squares, their expansions into the admissible graphs of size 5, and provide sketches of combinatorial proofs (note that the lemmas were initially obtained by solving the corresponding SDP problem).  We begin with the type $\tau_1$.

\begin{lemma} \label{lem41}
\[
\Delta_1=\bigg[\bigg[\left(M_2 + M_4 - M_1 - M_3\right)^2\bigg]\bigg]_{\tau_1}=\frac{1}{30}\cdot \left(2G_2+3G_3-G_5-G_8-4G_9-2G_{10}-5G_{11}\right) \ge 0.
\]
\end{lemma}
\begin{proof}[Sketch of proof]
Let $G=(V,E)$ be a graph on $n$ vertices. Define $\tau_1(G)=\{(x,y,z) \in V(G)^3: \{x,y\}, \{x,z\} \in E(G) \text{ and } \{y,z\}\not\in E(G)\}$. Every triple $(x,y,z)\in \tau_1(G)$ induces a copy of of the type $\tau_1$ in $G$, where vertex $x$ is labelled ``1'', vertex $y$ is labelled ``2'' and vertex $z$ is labelled ``3''.  Fix some $p = (x,y,z) \in \tau_1(G)$.  Note that $M_2$ and $M_4$ are flags where the unlabeled vertex is adjacent to $2$ but not $3$, while $M_1$ and $M_3$ are flags with the unlabeled vertex adjacent to $3$ but not $2$.  Hence we define
\[
d_p(v) \stackrel{def}{=} \left\{ \begin{array}{rl}
1, & \text{if } \{v, y\} \in E(G) \text{ but } \{v, z\} \not\in E(G), \\
-1, & \text{if } \{v, z\} \in E(G) \text{ but } \{v, y\} \not\in E(G), \\
0, & \text{otherwise,}
\end{array}\right.
\]
for each $v \in V(G)\setminus\{x, y, z\}$. If we denote by $F$ the flag induced by the labelled vertices $\{x, y, z\}$ together with the unlabelled vertex $v$, we have
\[
d_p(v) = \left\{ \begin{array}{rl}
1, & \text{if $F=M_2$ or $F=M_4$,}\\
-1, & \text{if $F=M_1$ or $F=M_3$,}\\
0, & \text{otherwise.}
\end{array}\right.
\]
Thus the combinatorial interpretation of the lemma is

\begin{align*}
\Delta_1(G) &= \frac{1}{3!\binom{n}{3}} \cdot\left[ \sum_{p=(x,y,z)\in \tau_1(G)} \frac{1}{2 \binom{n-3}{2} } \left(\sum_{\substack{v,w \not\in \{x,y,z\}\\v\ne w}} d_p(v) d_p(w)\right)\right] \\
    &= \frac{1}{120 \binom{n}{5}} \sum_{p = (x,y,z) \in \tau_1(G)} \sum_{\substack{v,w \notin \{x,y,z \} \\ v \ne w}} d_p(v) d_p(w) \ge o_{n\to\infty}(1).
\end{align*}
The proof that this summation is asymptotically non-negative is very simple, since
\[
\sum_{\substack{v,w \not\in \{x,y,z\} \\ v\ne w}} d_p(v) d_p(w) = \left(\sum_{v \not\in \{x,y,z\}} d_p(v)\right)^2 - \sum_{v\not\in \{x,y,z\}} d_p(v)^2,
\]
and
\[
\frac{1}{120 \binom{n}{5}} \cdot\left[ \sum_{p=(x,y,z)\in \tau_1(G)} \left(\sum_{v \not\in \{x,y,z\}} d_p(v)^2\right)\right] = O(1 /  n ).
\]
It remains to expand the products of the flags into admissible graphs of size 5, and thus show that $\Delta_1=\frac{1}{30}\cdot
\left(2G_2+3G_3-G_5-G_8-4G_9-2G_{10}-5G_{11}\right)$. For the sake of conciseness, we omit the full details of this calculation.  We show how to compute the coefficient of $G_{10}$, that is, $\Delta_1(G_{10})$; the other coefficients follow similarly.

In this case, the set $\{x, y, z, v, w\}$ spans a copy of $G_{10}$.

\begin{figure}[H]
\centering
\parbox{1.0in}{
\centering
\begin{tikzpicture}
  [scale=0.6,auto=left,every node/.style={circle, draw, fill=black!50,inner sep=0pt, minimum width=4pt}]
  \node (n1) at (18:1 cm)  {};
  \node (n2) at (90:1 cm)  [label=left:$y$]{};
  \node (n3) at (162:1 cm)  [label=left:$x$]{};
  \node (n4) at (234:1 cm)  [label=left:$z$]{};
  \node (n5) at (306:1 cm)  {};

  \foreach \from/\to in {n1/n2,n2/n3,n3/n4,n4/n5,n5/n1,n1/n3}
    \draw (\from) -- (\to);

\end{tikzpicture}
\caption*{$-4$}
}
\parbox{1.0in}{
\centering
\begin{tikzpicture}
  [scale=0.6,auto=left,every node/.style={circle, draw, fill=black!50,inner sep=0pt, minimum width=4pt}]
  \node (n1) at (18:1 cm)  {};
  \node (n2) at (90:1 cm)  [label=left:$z$]{};
  \node (n3) at (162:1 cm)  [label=left:$x$]{};
  \node (n4) at (234:1 cm)  [label=left:$y$]{};
  \node (n5) at (306:1 cm)  {};

  \foreach \from/\to in {n1/n2,n2/n3,n3/n4,n4/n5,n5/n1,n1/n3}
    \draw (\from) -- (\to);

\end{tikzpicture}
\caption*{$-4$}
}
\parbox{1.0in}{
\centering
\begin{tikzpicture}
  [scale=0.6,auto=left,every node/.style={circle, draw, fill=black!50,inner sep=0pt, minimum width=4pt}]
  \node (n1) at (18:1 cm)  [label=right:$y$]{};
  \node (n2) at (90:1 cm)  {};
  \node (n3) at (162:1 cm)  [label=left:$x$]{};
  \node (n4) at (234:1 cm)  [label=left:$z$]{};
  \node (n5) at (306:1 cm)  {};

  \foreach \from/\to in {n1/n2,n2/n3,n3/n4,n4/n5,n5/n1,n1/n3}
    \draw (\from) -- (\to);

\end{tikzpicture}
\caption*{$0$}
}
\parbox{1.0in}{
\centering
\begin{tikzpicture}
  [scale=0.6,auto=left,every node/.style={circle, draw, fill=black!50,inner sep=0pt, minimum width=4pt}]

  \node (n1) at (18:1 cm)  [label=right:$z$]{};
  \node (n2) at (90:1 cm)  {};
  \node (n3) at (162:1 cm)  [label=left:$x$]{};
  \node (n4) at (234:1 cm)  [label=left:$y$]{};
  \node (n5) at (306:1 cm)  {};

  \foreach \from/\to in {n1/n2,n2/n3,n3/n4,n4/n5,n5/n1,n1/n3}
    \draw (\from) -- (\to);

\end{tikzpicture}
\caption*{$0$}
}
\parbox{1.0in}{
\centering
\begin{tikzpicture}
  [scale=0.6,auto=left,every node/.style={circle, draw, fill=black!50,inner sep=0pt, minimum width=4pt}]

  \node (n1) at (18:1 cm)  {};
  \node (n2) at (90:1 cm)  {};
  \node (n3) at (162:1 cm)  [label=left:$y$]{};
  \node (n4) at (234:1 cm)  [label=left:$x$]{};
  \node (n5) at (306:1 cm)  [label=right:$z$]{};

  \foreach \from/\to in {n1/n2,n2/n3,n3/n4,n4/n5,n5/n1,n1/n3}
    \draw (\from) -- (\to);

\end{tikzpicture}
\caption*{$0$}
}
\parbox{1.0in}{
\centering
\begin{tikzpicture}
  [scale=0.6,auto=left,every node/.style={circle, draw, fill=black!50,inner sep=0pt, minimum width=4pt}]

  \node (n1) at (18:1 cm)  {};
  \node (n2) at (90:1 cm)  {};
  \node (n3) at (162:1 cm)  [label=left:$z$]{};
  \node (n4) at (234:1 cm)  [label=left:$x$]{};
  \node (n5) at (306:1 cm)  [label=right:$y$]{};

  \foreach \from/\to in {n1/n2,n2/n3,n3/n4,n4/n5,n5/n1,n1/n3}
    \draw (\from) -- (\to);

\end{tikzpicture}
\caption*{$0$}
}
\caption{Possible configurations of $p$ inside $G_{10}$ and corresponding contributions to $\Delta_1(G_{10})$.}
\label{lemma_cases}
\end{figure}

We have the following cases:
\begin{enumerate}
\item Vertex $x$ is one of the vertices of degree $3$. There are two choices of $x$ satisfying this condition. We have the following subcases:
\begin{enumerate}
\item Vertex $y$ is the vertex of degree $2$ of the triangle containing $x$ and $z$ is only neighbor of $x$ which is not adjacent to $y$. This configuration corresponds to the first graph in Figure \ref{lemma_cases}.  As one of the unlabeled vertices is adjacent to $y$ and not $z$, and the other is adjacent to $z$ and not $y$, both assignments of $v$ and $w$, we have $d_p(v) d_p(w) = -1$. As there are two choices for the pair $(v,w)$ and two choices for $x$, the total contribution for this configuration is $-4$.
\item The same configuration as above, but with the roles of $y$ and $z$ swapped. This configuration corresponds to the second graph in Figure \ref{lemma_cases} and its contribution is $-4$.
\item Vertex $y$ is the other vertex of degree $3$ and $z$ is the only neighbor of $x$ which is not adjacent to $y$. This configuration corresponds to the third graph in Figure \ref{lemma_cases}. For any possible choice of $v$ and $w$, we have $d_p(v)\cdot d_p(w) = 0$, hence the total contribution is $0$.
\item The same configuration as above, but with the roles of $y$ and $z$ swapped. This configuration corresponds to the fourth graph in Figure \ref{lemma_cases} and its contribution is $0$.
\end{enumerate}
\item Vertex $x$ is one of the vertices of degree $2$ not in the triangle. Again we have two choices of $x$ satisfying this condition. We also have the following subcases:
\begin{enumerate}
\item Vertex $y$ is the only neighbor of $x$ of degree $3$ and $z$ is the other neighbor. This configuration corresponds to the fifth graph in Figure \ref{lemma_cases}. For any possible choice of $v$ and $w$, we have $d_p(v)\cdot d_p(w) = 0$, hence the total contribution for this configuration is $0$.
\item The same configuration as above, but with the roles of $y$ and $z$ swapped. This configuration corresponds to the last graph in Figure \ref{lemma_cases} and its contribution is $0$.
\end{enumerate}
\end{enumerate}
When we sum the contributions we get $-8$, and hence the coefficient of $G_{10}$ is $\Delta_1(G_{10})=-\frac{8}{120}=-\frac{1}{15}$.

\end{proof}

We now consider the type $\tau_2$.

\begin{lemma} \label{lem42}
\begin{align*}
\Delta_2 &= \bigg[\bigg[\left(-3N_1-3N_2-3N_3-3N_4+2N_5+2N_6+2N_7+2N_8 \right)^2\bigg]\bigg]_{\tau_2} \\
    &= \frac{1}{10}\cdot\left(-24G_1-12G_2-24G_3-8G_5+28G_6+9G_7+9G_8+\right. \\
    & \qquad \qquad \left.18G_9+9G_{10}-12G_{12}+16G_{13}+40G_{14}\right) \ge 0.
\end{align*}
\end{lemma}
\begin{proof}[Sketch of proof]
Let $G=(V,E)$ be a graph on $n$ vertices. Define $\tau_2(G)=\{(x,y,z) \in V(G)^3: \{x,y\}, \{x,z\}, \{y,z\} \in E(G)\}$. Every triple $(x,y,z)\in \tau_2(G)$ induces a copy of of the type $\tau_2$ in $G$, where vertex $x$ is labelled ``1'', vertex $y$ is labelled ``2'' and vertex $z$ is labelled ``3''. Fix $p=(x,y,z) \in \tau_2(G)$.  Note that the flags $N_i$ for $1 \le i \le 4$ are those where the unlabeled vertex has at most one neighbour in the triangle $\tau_2$, while in the flags $N_i$ for $5 \le i \le 8$, the unlabeled vertex has at least two neighbours in $\tau_2$.  This motivates the definition
\[
d_p(v) \stackrel{def}{=} \left\{ \begin{array}{rl}
-3, & \text{if $v$ is connected to at most one vertex in $\{x,y,z\}$,} \\
2, & \text{otherwise,} \\
\end{array}\right.
\]
for each $v \in V(G)\setminus\{x, y, z\}$. The combinatorial interpretation of the lemma is
\[
\Delta_2(G)=\frac{1}{5! \binom{n}{5}} \left[ \sum_{p=(x,y,z)\in \tau_2(G)} \left(\sum_{\substack{v,w \not\in \{x,y,z\}\\v\ne w}} d_p(v) d_p(w)\right)\right] \ge o_{n\to\infty}(1).
\]
As in Lemma \ref{lem41}, this is easily seen to be asymptotically positive.  We omit the computation of $\Delta_2(G_i)$ for $i=1,2,\ldots, 14$, which can be performed as in the proof of the previous lemma.
\end{proof}

Finally we consider the \emph{dot} type.  Note that in this case the positive semi-definite matrix takes the form of a sum of three squares.

\begin{lemma} \label{lem43}
\begin{align*}
\Delta_3 &= \bigg[\bigg[(Z_1 - 2Z_2)^2+ \frac{1}{16}\cdot \left(6Z_2 - 7Z_3 + 8Z_4 - 6Z_5\right)^2+\frac{11}{80}\cdot \left(2Z_2 + 3Z_3 - 2Z_5\right)^2\bigg]\bigg]_{dot} \\
    &=\frac{1}{150}\left(204 G_1 - 118 G_2 +54G_3 + 60G_4 -17 G_5 + 42G_6 - 144 G_7 - 94 G_8+\right. \\
    & \qquad \qquad \left.2 G_9 -64 G_{10} + 160 G_{11} -258G_{12} -281G_{13}+420G_{14}\right) \ge 0.
\end{align*}
\end{lemma}
\begin{proof}
We omit the proof, noting that the calculations involved are very similar to those in the previous lemmas.
\end{proof}

We are now in a position to combine the lemmas to obtain a bound on the minimum density of $4$-cliques in admissible graphs.  In what follows, $K_4$ represents the clique on four vertices, while $C_4$ denotes a cycle on four vertices.

\begin{thm}\label{thm44}
\begin{align*}
K_4 - 2 \Delta_1 - \frac{2}{25} \Delta_2 -\frac{1}{5} \Delta_3 &= \frac{3}{25} + \frac{1}{30} G_5 + \frac{2}{75} G_{10}
+\frac{24}{75} G_{12} + \frac{19}{150} G_{13} \\
&= \frac{3}{25} + \frac{1}{30} G_5+ \frac{2}{15} C_4 + \frac{4}{15} G_{12} + \frac{1}{10} G_{13}.
\end{align*}
\end{thm}

\begin{proof}
We first expand the graphs $K_4$ and $C_4$ into admissible graphs of size $5$.  A straightforward calculation gives $K_4 = \frac{1}{5} (G_1 + G_3 + G_4 + 2G_6 + 5 G_{14}),$ and $C_4 = \frac{1}{5} (G_{10} + 2G_{12} + G_{13}).$  Note that the density of graphs on $k$ vertices is measured with respect to $\binom{n}{k}$, and so the normalization factor of $\frac{1}{5}$ appears when expanding graphs on four vertices to graphs on five vertices.  Now we use Lemmas \ref{lem41}, \ref{lem42} and \ref{lem43} to expand $\Delta_1$, $\Delta_2$ and $\Delta_3$ into the graphs $G_i$.  Noting that $\sum_i G_i = 1$, we can replace $\frac{3}{25} \sum_i G_i$ with $\frac{3}{25}$, which results in the above theorem.
\end{proof}

We conclude this section by using the above theorem to deduce some structural information about extremal graphs.  Recall that $t_4(G)$ denotes the number of $4$-cliques in $G$, while for any graph $H$, $t_H(G)$ counts the number of induced copies of $H$ in $G$.

\begin{cor} \label{cor45}
Suppose $G$ is a graph on $n$ vertices with $t_4(G) = \left( \frac{3}{25} + o(1) \right) \binom{n}{4}$.  Then
\begin{itemize}
    \item[(i)] $t_{G_5}(G) = o(n^5)$,
    \item[(ii)] $t_{C_4}(G) = o(n^4)$, and
    \item[(iii)] all but $o(n)$ vertices of $G$ have degree $(\frac{3}{5} + o(1))n$.
\end{itemize}
\end{cor}

\begin{proof}

Applying Theorem \ref{thm44} to $G$, we have
\[ d(K_4; G) - 2 \Delta_1 (G) - \frac{2}{25} \Delta_2 (G) - \frac{1}{5} \Delta_3(G) = \frac{3}{25} + \frac{1}{30} d(G_5; G) + \frac{2}{15} d(C_4; G) + \frac{4}{15} d(G_{12}; G) + \frac{1}{10} d(G_{13}; G). \]
In particular, using the asymptotic non-negativity of $\Delta_i(G)$, we have
\[ d(K_4; G) \ge \frac{3}{25} + \frac{1}{30} d(G_5; G) + \frac{2}{15} d(C_4; G) + \frac{1}{5} \Delta_3(G) + o(1). \]
Thus if $d(K_4; G) = \frac{3}{25} + o(1)$, we must have $d(G_5; G) = d(C_4; G) = \Delta_3(G) = o(1)$.  This immediately gives $t_{G_5}(G) = o(n^5)$ and $t_{C_4}(G) = o(n^4)$, and so it remains to justify (iii).  We have
\[ \Delta_3(G) = \bigg[ \bigg[ (Z_1 - 2 Z_2)^2 + \frac{1}{16} (6 Z_2 - 7 Z_3 + 8 Z_4 - 6 Z_5)^2 + \frac{11}{80} ( 2 Z_2 + 3 Z_3 - 2 Z_5)^2 \bigg] \bigg]_{dot} = o(1). \]

For every vertex $v$, let $F_v$ be the \emph{dot}-flag obtained from $G$ by labeling the vertex $v$ with $1$.  By definition of the averaging operator, $\Delta_3(G)$ is the average over vertices $v$ of the corresponding flag densities in $F_v$.  The expression is a sum of squares, and thus will be asymptotically non-negative.  Since the average is $o(1)$, the expression must be $o(1)$ for all but $o(n)$ vertices.  In particular, for these vertices we have
\begin{align*}
    p_{dot}(Z_1; F_v) - 2 p_{dot}(Z_2; F_v) &= o(1), \\
    6 p_{dot}(Z_2; F_v) - 7 p_{dot}(Z_3; F_v) + 8 p_{dot} (Z_4; F_v) - 6 p_{dot} (Z_5; F_v) &= o(1), \textrm{ and} \\
    2 p_{dot}(Z_2; F_v) + 3 p_{dot}(Z_3; F_v) - 2p_{dot}(Z_5; F_v) &= o(1).
\end{align*}

Since the sum of the flag densities must be $1$, we also have
\begin{align*}
    p_{dot}(Z_1; F_v) + p_{dot}(Z_2; F_v) + p_{dot}(Z_3; F_v) + p_{dot}(Z_4; F_v) + p_{dot}(Z_5; F_v) &= 1.
\end{align*}

Finally, recall from Equation (\ref{stability_equation})     in Section \ref{flag_intro} that $4 Z_2 \cdot ( Z_3 + Z_5 ) - (Z_4 + Z_1)^2 = 0$.  Applying this to $F_v$, we have
\begin{align*}
    4p_{dot}(Z_2; F_v) \left( p_{dot}(Z_3; F_v) + p_{dot}(Z_5; F_v) \right) - \left( p_{dot}(Z_4; F_v) + p_{dot}(Z_1; F_v) \right)^2 &= o(1).
\end{align*}

This gives us a system of five equations in the five variables
$p_{dot}(Z_i; F_v)$.  The first four equations form a linear system
of full rank, which we can use to express all the variables in terms
of $p_{dot}(Z_5; F_v)$. Substituting these terms into the fifth
equation gives a quadratic equation in $p_{dot}(Z_5; F_v)$, which
results in two solutions, namely $ \left( p_{dot}(Z_i; F_v)
\right)_{i=1}^5 = \left( \frac{8}{25}, \frac{4}{25}, \frac{2}{25},
\frac{4}{25}, \frac{7}{25} \right) + o(1)$ or $\left( \frac{1}{2},
\frac{1}{4}, 0, 0, \frac{1}{4} \right) + o(1)$.

We now show that the second solution implies a large number of $4$-cliques.  Indeed, suppose $v \in V$ was a vertex with $  \left( p_{dot}(Z_i; F_v) \right)_{i=1}^5 = \left( \frac{1}{2}, \frac{1}{4}, 0, 0, \frac{1}{4} \right) + o(1)$.  Recall from Equation (\ref{edgeexpansion}) in Section 3 we have $\rho = \frac{1}{2}Z_1 + Z_3 + \frac{1}{2} Z_4 + Z_5$, where $\rho$ is the \emph{dot}-flag of size 2 corresponding to an edge.  Applying this to the flag $F_v$, we deduce that the degree of $v$ is $\left( \frac{1}{2} \cdot \frac{1}{2} + \frac{1}{4} + o(1) \right) n = \frac{1}{2} n + o(n) $.  Thus there are $\frac{1}{2} n + o(n) $ vertices $v$ is not adjacent to, and since $G$ is $\overline{K_3}$-free, these vertices must form a clique.  This clique contains $\binom{ \frac{1}{2}n + o(n)}{4} \sim \frac{1}{16} \binom{n}{4}$ $4$-cliques.  Consider now the neighborhood of $v$.  Since $p_{dot}(Z_3; F_v) = o(1)$, it follows that the neighborhood is missing at most $o(n^2)$ edges.  Hence the number of $4$-cliques in the neighborhood of $v$ is $\binom{ \frac{1}{2} n + o(n) }{4} - o(n^4) \sim \frac{1}{16} \binom{n}{4}$.  Thus we have $t_4(G) \ge \left( \frac{1}{8} + o(1) \right) \binom{n}{4}$, which contradicts our assumption that $t_4(G) = \left( \frac{3}{25} + o(1) \right) \binom{n}{4}$.

Hence for almost all vertices $v$, we have $( p_{dot}(Z_i; F_v) )_{i=1}^5 = \left( \frac{8}{25}, \frac{4}{25}, \frac{2}{25}, \frac{4}{25}, \frac{7}{25} \right) + o(1)$.  Applying Equation (\ref{edgeexpansion}), we deduce that the degree of $v$ is $\left( \frac{1}{2} \cdot \frac{8}{25} + \frac{2}{25} + \frac{1}{2} \cdot \frac{4}{25} + \frac{7}{25} + o(1) \right) n = \left( \frac{3}{5} n + o(1)\right) n$, as claimed.

\end{proof}

\subsection{The stability analysis} \label{c5stab}

We will now use the results of the preceding section to show that, for sufficiently large $n$, a blow-up of $C_5$ is the unique extremal graph for the $(4,3)$-problem.  Recall that in a blow-up, we replace every vertex with a clique, and every edge with a complete bipartite graph.  Hence a blow-up of $C_5$ consists of five disjoint sets of vertices $V_i$, with $V_i \cup V_{i+1}$ a clique for all $1 \le i \le 5$, and no edges between $V_i$ and $V_{i+2}$ for all $1 \le i \le 5$ (throughout this section, indices will be taken modulo 5).

Suppose $G$ is a $\overline{K_3}$-free graph on $n$ vertices with
the minimal number of $4$-cliques.  Our proof consists of three
steps.  We first use the results of Corollary \ref{cor45} to deduce
that $G$ is close to being a blow-up of $C_5$ (note that this holds
not only for an extremal graph, but for any family of graphs that is
asymptotically optimal).  In the second step we use the minimality
of $G$ to show that $G$ must in fact be a blow-up of $C_5$ with
asymptotically equal parts.  Finally, we solve an integer
optimization problem to determine the size of the parts of $G$
exactly.

\medskip

Recall that from Corollary \ref{cor45}, we have that if $n$ is sufficiently large, and $G$ is an extremal graph on $n$ vertices, then $t_4(G) = \frac{3}{25} \binom{n}{4} + o(n^4)$, $t_{C_4}(G) = o(n^4)$, $t_{G_5}(G) = o(n^5)$, and all but $o(n)$ vertices of $G$ have degree $ \frac{3}{5} n + o(n)$.  From this we shall deduce that $G$ is almost a blow-up of $C_5$.  To this end, we introduce some definitions.  Given subsets $A, B \subset V(G)$, we say $A$ is an \emph{almost clique} if all but $o(n^2)$ pairs in $A$ are adjacent, and we say $(A,B)$ is \emph{almost complete} (\emph{almost empty}) if all but $o(n^2)$ pairs in $A \times B$ are adjacent (nonadjacent).  Finally, we define a triple $\{ a, b, c \} \in V(G)$ to be \emph{typical} if:
\begin{itemize}
    \item[(i)] $\{a , b \} \notin E(G)$, $c \in N(a) \cap N(b)$, $d(a), d(b), d(c) = \frac{3}{5}n + o(n)$,
    \item[(ii)] $\{a , b\}$ is contained in $o(n^2)$ copies of $C_4$,
    \item[(iii)] $\{ a, b, c \}$ is contained in $o(n)$ copies of $C_4$, and
    \item[(iv)] $\{a, b, c \}$ is contained in $o(n^2)$ copies of $G_5$.
\end{itemize}

Note that $G[\{a,b,c\}]$ is an induced path of length $2$.  As all but $o(n)$ vertices are of degree $\frac{3}{5}n + o(n)$, it is easy to see that there are $\Omega(n^3)$ induced paths of length $2$ in $G$.  As Corollary \ref{cor45} asserts that $t_{C_4}(G) = o(n^4)$ and $t_{G_5}(G) = o(n^5)$, it follows that almost all induced paths of length $2$ are typical.  We will now use the neighborhoods of $\{a , b, c \}$ to define the parts corresponding to the blow-up of $C_5$.  In particular, we define
\begin{align*}
    &V_1 = N(a) \cap N(b), V_2 = \{a\} \cup \left( N(a) \cap \overline{N(b)} \cap N(c) \right), V_3 = N(a) \cap \overline{N(b)} \cap \overline{N(c)}, \\
    &V_4 = \overline{N(a)} \cap N(b) \cap \overline{N(c)}, \textrm{ and } V_5 = \{ b \} \cup \left( \overline{N(a)} \cap N(b) \cap N(c) \right).
\end{align*}

We now make some preliminary observations about the sets $V_i$.
Clearly, by definition, the sets are disjoint.  Moreover, since
$\alpha(G) \le 2$, and $\{ a, b \} \notin E(G)$, we must have $N(a)
\cup N(b) = V(G) \setminus \{ a, b \}$, and so $\cup_i V_i = V(G)$.
Similarly, for any vertex $v \in V(G)$, $\overline{N(v)}$ must
induce a clique, as any non-edge in $\overline{N(v)}$ forms an
independent set of size three with $v$.  Thus $V_2 \cup V_3$, $V_3
\cup V_4$, and $V_4 \cup V_5$ are (actual) cliques.  Finally, note
that if $u, v \in V_1$ are such that $\{ u, v \} \notin E(G)$, then
the set $\{ a, b, u, v \}$ induces a copy of $C_4$.  Since $\{a , b,
c \}$ was chosen to be a typical triple, properties (ii) and (iii)
imply that $V_1$ is an almost clique, and $c$ is adjacent to all but
$o(n)$ vertices in $V_1$.

We can also obtain some relations regarding the sizes of these parts.  By property (i) of typical triples, we have $d(a), d(b), d(c) = \frac{3}{5} n + o(n)$.  Since $N(a) \cup N(b) = V(G) \setminus \{a, b \}$, we have $|V_1| = |N(a) \cap N(b)| = |N(a)| + |N(b)| - |N(a) \cup N(b)| = \frac{1}{5} n + o(n)$.  Moreover, as $N(a) \cup \{a \} = V_1 \cup V_2 \cup V_3$, $N(b) \cup \{ b \} = V_1 \cup V_4 \cup V_5$, $\overline{N(c)} \setminus V_1 = V_3 \cup V_4$, and $c$ has $o(n)$ non-neighbors in $V_1$, we deduce
\[ |V_2| + |V_3| = \frac{2}{5} n + o(n), |V_3| + |V_4| = \frac{2}{5} n + o(n), \textrm{ and } |V_4| + |V_5| = \frac{2}{5}n + o(n), \]
which also imply $|V_2| + |V_5| = \frac{2}{5} n + o(n)$.

\medskip

We are beginning to uncover the approximate $C_5$-blow-up structure of $G$.  Recall that we have shown that $V_2 \cup V_3$, $V_3 \cup V_4$ and $V_4 \cup V_5$ are cliques, while $V_1$ is an almost clique.  We will establish the relations between the remaining parts by showing:
\begin{itemize}
    \item $(V_i, V_{i+2})$ is almost empty for any $1 \le i \le 5$, and
    \item $(V_1, V_2)$ and $(V_1, V_5)$ are almost complete.
\end{itemize}
We start by showing that $(V_1, V_3)$ is almost empty.  For any $u \in V_1 \cap N(c)$ and $v \in V_3$, if $\{ u, v \} \in E(G)$, then the set $\{ a, b, c, u, v \}$ induces a copy of $G_5$.  As $\{ a, b, c \}$ is a typical triple, property (iv) implies that there are at most $o(n^2)$ copies of $G_5$ containing $\{ a, b, c \}$, and so there are at most $o(n^2)$ edges between $V_1 \cap N(c)$ and $V_3$.  Since $c$ is adjacent to all but $o(n)$ vertices in $V_1$, this shows that $(V_1, V_3)$ is almost empty.  By the symmetry between $a$ and $b$ (and hence $V_3$ and $V_4$), it follows that $(V_1, V_4)$ is also almost empty.

Now consider the vertices in $V_1$.  By Corollary \ref{cor45}, all but $o(n)$ of these vertices have degree $\frac{3}{5}n + o(n)$.  Since $(V_1, V_3 \cup V_4)$ is almost empty, it follows that all but $o(n)$ vertices in $V_1$ have $o(n)$ edges to $V_3 \cup V_4$.  Hence, since $|V_1| + |V_2| + |V_5| = \frac{3}{5}n + o(n)$, it follows that $V_1$ is almost complete to $V_1 \cup V_2 \cup V_5$.  In particular, $(V_1, V_2)$ and $(V_1, V_5)$ are almost complete.

Next consider the vertices in $V_2$.  We have established that $(V_2, V_1 \cup V_2 \cup V_3)$ is almost complete.  Once again, using the restriction on the degrees, and the fact that $|V_1| + |V_2| + |V_3| = \frac{3}{5}n + o(n)$, we deduce that $(V_2, V_4)$ and $(V_2, V_5)$ are almost empty.  Symmetry implies $(V_5, V_2)$ and $(V_5, V_3)$ are almost empty as well, as claimed.

\medskip

At this point we have determined the global structure of $G$, in which each part $V_i$ corresponds approximately to the blow-up of a vertex in $C_5$.  We now wish to show that $G$ is an exact blow-up of $C_5$, with parts of size $\frac{1}{5}n + o(n)$.

\medskip

In order to do so, we shall require greater control over the adjacency of individual vertices, and not just the parts $V_i$.  With this in mind, for each $1 \le i \le 5$, we define a vertex $v \in V_i$ to be \emph{bad} if $v$ has $\Omega(n)$ non-neighbors in $V_{i-1} \cup V_i \cup V_{i+1}$ or $\Omega(n)$ neighbors in $V_{i+2} \cup V_{i+3}$.  Since for each $i$ we have that $V_i \cup V_{i+1}$ is an almost clique and $(V_i, V_{i+2})$ is almost empty, it follows that there are $o(n)$ bad vertices.  We clean up the partition of $V(G)$ by removing bad vertices from each $V_i$ and placing them in a set $U$.  This results in a partition $V(G) = V_1 \cup \hdots \cup V_5 \cup U$ satisfying:
\begin{itemize}
    \item[(1)] for any $1 \le i \le 5$ and vertex $v \in V_i$, $v$ is adjacent to all but $o(n)$ vertices in $V_{i-1} \cup V_i \cup V_{i+1}$, and $v$ is not adjacent to all but $o(n)$ vertices in $V_{i+2} \cup V_{i+3}$, and
    \item[(2)] $V_2 \cup V_3$, $V_3 \cup V_4$, $V_4 \cup V_5$ are cliques, and
    \item[(3)] $|V_1| = \frac{1}{5}n + o(n)$, $|V_2 \cup V_3|, |V_3 \cup V_4|, |V_4 \cup V_5| = \frac{2}{5} n + o(n)$, and $|U| = o(n)$.
\end{itemize}

The following proposition asserts that in an asymptotically optimal graph, the above conditions imply that the almost cliques are, in fact, true cliques, and that the parts are asymptotically equal.  This will in turn allow us to completely determine the structure of extremal graphs.

\begin{prop}\label{c5cliques}
If $V_1, V_2, \hdots, V_5$ satisfy $(1), (2)$ and $(3)$, then for any $1 \le i \le 5$, $V_i \cup V_{i+1}$ is a clique, and $|V_i| = \frac{1}{5}n + o(n)$.
\end{prop}

\begin{proof}
We already know from $(2)$ that many of the pairs of neighboring parts are cliques.  It remains to show that $V_1 \cup V_2$ and $V_5 \cup V_1$ are both cliques.  We first show that $V_1$ is a clique.  Suppose for contradiction that there are nonadjacent vertices $u, v \in V_1$.  Since $\alpha(G) \le 2$, we must have $V_3 \cup V_4 \subset N(u) \cup N(v)$.  By $(3)$ we have $|V_3 \cup V_4| = \frac{2}{5}n + o(n)$, and so either $u$ or $v$ must have at least $\frac{1}{5} n + o(n)$ neighbors in $V_3 \cup V_4$.  However, this contradicts $(1)$.  Thus $V_1$ is a clique.

We now claim that if $(V_1, V_2)$ is not complete, we must have $|V_4|= o(n)$.  Indeed, suppose $u \in V_1$ and $v \in V_2$ are not adjacent.  Since $\alpha(G) \le 2$, we must have $V_4 \subset N(u) \cup N(v)$.  By $(1)$, both $u$ and $v$ have $o(n)$ neighbors in $V_4$, which implies $|V_4| = o(n)$.  By symmetry, if $(V_1, V_5)$ is not complete, we must have $|V_3| = o(n)$.

Suppose now that one of these sets, say $V_4$, is of size $o(n)$.  Using $(3)$, we must have $|V_3| = |V_5| = \frac{2}{5}n + o(n)$, and $|V_2| = o(n)$.  Since $|V_3| \ne o(n)$, it follows that $(V_1, V_5)$ is complete.  Thus $G$ has two large disjoint cliques: $V_3$ of size $\frac{2}{5} n + o(n)$, and $V_1 \cup V_5$ of size $\frac{3}{5} n + o(n)$.  This gives
\[ t_4(G) \ge \binom{\frac{2}{5}n + o(n)}{4} + \binom{ \frac{3}{5}n + o(n)}{4} \sim \frac{97}{625} \binom{n}{4} + o(n^4) > \frac{3}{25} \binom{n}{4}, \]
contradicting the asymptotic optimality of $G$.  Hence $(V_1, V_2)$ and $(V_1, V_5)$ must be complete, which implies that $V_1 \cup V_2$ and $V_1 \cup V_5$ are cliques.

Finally, we show that all parts have size $\frac{1}{5} n + o(n)$.  Recall we already have $|V_1| = \frac{1}{5} n + o(n)$.  Since $|V_3| + |V_4| = \frac{2}{5}n + o(n)$, we may by symmetry assume $|V_3| \ge \frac{1}{5} n + o(n)$.  Corollary \ref{cor45} implies there is some vertex of $V_3$ whose degree is $\frac{3}{5}n + o(n)$.  By $(1)$, this implies $|V_2| + |V_3| + |V_4| = \frac{3}{5} n + o(n)$.  As $|V_3| + |V_4| = \frac{2}{5}n + o(n)$, this implies $|V_2| = \frac{1}{5} n + o(n)$.  Combined with the equations in $(3)$, this gives $|V_i| = \frac{1}{5}n + o(n)$ for all $2 \le i \le 5$.
\end{proof}

We now turn our attention to the set $U$ of bad vertices.  In particular, we will show that in an extremal graph, each $u \in U$ can be reintroduced into some part $V_i$ in a way that is consistent with $(1)$ and Proposition \ref{c5cliques}.  Since $|U| = o(n)$, we can repeat this process without affecting $(1)$ or Proposition \ref{c5cliques}, and thus we can eliminate the set $U$.

\begin{prop} \label{c5nobad}
For every $u \in U$, there is some $i = i(u)$ such that $V_{i-1} \cup V_i \cup V_{i+1} \subset N(u)$, and $u$ has $o(n)$ neighbors in $V_{i+2} \cup V_{i+3}$.
\end{prop}

\begin{proof}
Fix $u \in U$.  We begin with a simple claim.  For any $1 \le j \le 5$, if there is some $v \in V_j$ such that $u$ is not adjacent to $v$, then $u$ is adjacent to all but $o(n)$ vertices in $V_{j+2} \cup V_{j+3}$.  Indeed, as $\alpha(G) \le 2$, we must have $V_{j+2} \cup V_{j+3} \subset N(u) \cup N(v)$.  However, $v$ is adjacent to $o(n)$ vertices in $V_{j+2} \cup V_{j+3}$, and so the claim follows.

\medskip

Now suppose there is no $i$ such that $V_{i-1} \cup V_i \cup V_{i+1} \subset N(u)$.  This implies there is an $i$ such that $u$ is not adjacent to some vertices in both $V_{i-3}$ and $V_{i-1}$.  Applying the previous claim, it follows that $u$ is adjacent to all but $o(n)$ vertices in $V_{i-1} \cup V_i \cup V_{i+1} \cup V_{i+2}$.

In this case, remove all edges between $u$ and $V_{i+2}$, and add any missing edges between $u$ and $V_{i-1} \cup V_i \cup V_{i+1} \cup U$.  It is easy to see that we still have $\alpha(G) \le 2$.  As $u$ had $\frac{1}{5}n + o(n)$ neighbors in $V_{i+2}$, which is a clique, we have removed at least $\binom{ \frac{1}{5}n + o(n) }{3} = \Omega(n^3)$ $4$-cliques.  On the other hand, we have only added $o(n)$ edges, and so created $o(n^3)$ new $4$-cliques.  Thus we have reduced the number of $4$-cliques, which contradicts the extremality of $G$.

\medskip

Thus there must be some $i = i(u)$ such that $V_{i-1} \cup V_i \cup V_{i+1} \subset N(u)$.  It remains to show that $u$ has $o(n)$ neighbors in $V_{i+2} \cup V_{i+3}$.  Suppose for contradiction that $u$ has $\Omega(n)$ neighbors in $V_{i+2} \cup V_{i+3}$.  As $V_{i+2} \cup V_{i+3}$ is a clique, these neighbors form $\Omega(n^3)$ $4$-cliques with $u$.  Instead, we could remove all edges between $u$ and $V_{i+2} \cup V_{i+3}$.  To prevent the formation of an independent set of size $3$, we add all edges between $u$ and $U$.  This introduces $o(n)$ new edges, and thus $o(n^3)$ new $4$-cliques, while maintaining $\alpha(G) \le 2$.  Thus the number of $4$-cliques is reduced, again contradicting the minimality of $G$.  This completes the proof.

\end{proof}

Given any $u \in U$, we can apply Proposition \ref{c5nobad} to add $u$ to $V_{i(u)}$.  Repeat this process until $U$ is empty.  In this case we have a partition $V(G) = V_1 \cup \hdots \cup V_5$ such that for every $1 \le i \le 5$, $|V_i| = \frac{1}{5}n + o(n)$ and $V_i \cup V_{i+1}$ is a clique.

In order to conclude that $G$ is a blow-up of $C_5$, it remains to show that there are no edges between $V_{i-1}$ and $V_{i+1}$ for any $i$.  Suppose to the contrary there is an edge between some $v \in V_{i-1}$ and $w \in V_{i+1}$.  Note that when $n$ is large, we must have $|V_i| = \frac{1}{5} n + o(n) \ge 2$.  For any $x, y \in V_i$, $\{v, w, x, y\}$ is a $4$-clique.  Thus removing the edge $\{ v, w\}$ reduces the number of $4$-cliques without increasing the independence number.  Hence in an extremal graph, there are no edges between $V_{i-1}$ and $V_{i+1}$ for any $i$, and thus $G$ is indeed a blow-up of $C_5$ with parts of size $\frac{1}{5}n + o(n)$.

\medskip

We now seek to determine the sizes of the sets $V_i$ exactly.  Noting that $V_i \cup V_{i+1}$ is a clique for each $i$, it is easily verified that
\[ t_4(G) = \sum_{i=1}^5 \binom{|V_i \cup V_{i+1}|}{4} - \sum_{i=1}^5 \binom{|V_i|}{4}. \]
Define $y_i = |V_{2i- 1} \cup V_{2i}|$ for all $1 \le i \le 5$.  In $\sum y_i$, each vertex is counted twice, so we have $\sum y_i = 2n$.  Moreover, as $|V_i| = \frac{1}{5}n + o(n)$, we have $y_i = \frac{2}{5}n + o(n)$.  Finally, as $n - y_i - y_{i+1} = n - |V_{2i-1}| - |V_{2i}| - |V_{2i+1}| - |V_{2i+2}| = |V_{2i-2}|$, we can rewrite the above expression as
\[ t_4(G) = \sum_{i=1}^5 \binom{y_i}{4} - \sum_{i=1}^5 \binom{n - y_i - y_{i+1}}{5}. \]
Thus to find the extremal graph, we must minimize the above expression over integer values of $y_i$ subject to the conditions given earlier.  The solution is given by Lemma \ref{c5opt}, which we prove in Appendix \ref{app_intopt}.

\begin{lemma} \label{c5opt}
Let $\varepsilon > 0$ be sufficiently small, and $n$ sufficiently large.  Consider the function
\[ g(y_1, y_2, y_3, y_4, y_5) = \sum_{i=1}^5 \binom{y_i}{5} - \sum_{i=1}^5 \binom{n - y_i - y_{i+1}}{4}. \]
Subject to the constraints that the $y_i$ be integers satisfying $\sum_{i=1}^5 y_i = 2n$ and $\left| y_i - \frac{2}{5} n \right| < \varepsilon n$, $g$ is uniquely (up to cyclic permutation of the variables) minimized when the $y_i$ take values $\left \lfloor \frac{2n}{5} \right \rfloor$ and $\left \lceil \frac{2n}{5} \right \rceil$ in ascending order.
\end{lemma}

From Lemma \ref{c5opt}, we see the minimum occurs when $y_i = \left \lceil \frac{2n + i - 1}{5} \right \rceil$ for $1 \le i \le 5$.  Solving for $|V_i|$, we have that the unique extremal graph on $n$ vertices is the blow-up of $C_5$ to $n$ vertices such that:
\begin{itemize}
    \item when $n = 5k$, $|V_i| = k$ for all $i$,
    \item when $n = 5k + 1$, $|V_1| = |V_2| = k$, $|V_3| = |V_5| = k+1$, and $|V_4| = k-1$,
    \item when $n = 5k + 2$, $|V_1| = |V_2| = |V_4| = k$, and $|V_3| = |V_5| = k+1$,
    \item when $n = 5k + 3$, $|V_1| = |V_2| = |V_4| = k+1$, and $|V_3| = |V_5| = k$, and
    \item when $n = 5k + 4$, $|V_1| = |V_2| = k+ 1$, $|V_3| = |V_5| = k$, and $|V_4| = k+2$.
\end{itemize}

\section{The $(3,4)$-problem} \label{sec34}

In this section we solve the $(3,4)$-problem, and prove that Erd\H{o}s' conjecture holds for this case.  Recall that this entails showing that amongst all graphs of independence number less than four, $\overline{T_{n,3}}$, a disjoint union of three nearly-equal cliques, minimizes the number of triangles.

In the first subsection we list our flag algebra results, which give the asymptotic minimum number of triangles to be $\frac{1}{9} \binom{n}{3}$.  In the second subsection we use the structural information obtained to determine the value of $f(n,3,4)$ exactly.  We also analyze the structure of extremal graphs, and show they must contain $\overline{T_{n,3}}$.

\subsection{Getting the asymptotic result and densities}

We begin by presenting the 29 admissible - that is, $\overline{K_4}$-free - graphs of size 5, followed by the three types and associated flags used in the proof.

\begin{figure}[H]
\centering
\parbox{0.7in}{
\centering
\begin{tikzpicture}
  [scale=0.6,auto=left,every node/.style={circle, draw, fill=black!50,inner sep=0pt, minimum width=4pt}]
  \node (n1) at (18:1 cm) {};
  \node (n2) at (90:1 cm)  {};
  \node (n3) at (162:1 cm)  {};
  \node (n4) at (234:1 cm)  {};
  \node (n5) at (306:1 cm)  {};

  \foreach \from/\to in {n3/n4,n3/n5,n4/n5}
    \draw (\from) -- (\to);

\end{tikzpicture}
\caption*{$G_1$} }
\parbox{0.7in}{
\centering
\begin{tikzpicture}
  [scale=0.6,auto=left,every node/.style={circle, draw, fill=black!50,inner sep=0pt, minimum width=4pt}]
  \node (n1) at (18:1 cm) {};
  \node (n2) at (90:1 cm)  {};
  \node (n3) at (162:1 cm)  {};
  \node (n4) at (234:1 cm)  {};
  \node (n5) at (306:1 cm)  {};

  \foreach \from/\to in {n2/n5,n3/n4,n3/n5,n4/n5}
    \draw (\from) -- (\to);

\end{tikzpicture}
\caption*{$G_2$} }
\parbox{0.7in}{
\centering
\begin{tikzpicture}
  [scale=0.6,auto=left,every node/.style={circle, draw, fill=black!50,inner sep=0pt, minimum width=4pt}]
  \node (n1) at (18:1 cm) {};
  \node (n2) at (90:1 cm)  {};
  \node (n3) at (162:1 cm)  {};
  \node (n4) at (234:1 cm)  {};
  \node (n5) at (306:1 cm)  {};

  \foreach \from/\to in {n2/n4,n3/n5,n4/n5}
    \draw (\from) -- (\to);

\end{tikzpicture}
\caption*{$G_3$} }
\parbox{0.7in}{
\centering
\begin{tikzpicture}
  [scale=0.6,auto=left,every node/.style={circle, draw, fill=black!50,inner sep=0pt, minimum width=4pt}]
  \node (n1) at (18:1 cm) {};
  \node (n2) at (90:1 cm)  {};
  \node (n3) at (162:1 cm)  {};
  \node (n4) at (234:1 cm)  {};
  \node (n5) at (306:1 cm)  {};

  \foreach \from/\to in {n2/n4,n2/n5,n3/n4,n3/n5,n4/n5}
    \draw (\from) -- (\to);

\end{tikzpicture}
\caption*{$G_4$} }
\parbox{0.7in}{
\centering
\begin{tikzpicture}
  [scale=0.6,auto=left,every node/.style={circle, draw, fill=black!50,inner sep=0pt, minimum width=4pt}]
  \node (n1) at (18:1 cm) {};
  \node (n2) at (90:1 cm)  {};
  \node (n3) at (162:1 cm)  {};
  \node (n4) at (234:1 cm)  {};
  \node (n5) at (306:1 cm)  {};

  \foreach \from/\to in {n2/n3,n4/n5}
    \draw (\from) -- (\to);

\end{tikzpicture}
\caption*{$G_5$} }
\parbox{0.7in}{
\centering
\begin{tikzpicture}
  [scale=0.6,auto=left,every node/.style={circle, draw, fill=black!50,inner sep=0pt, minimum width=4pt}]
  \node (n1) at (18:1 cm) {};
  \node (n2) at (90:1 cm)  {};
  \node (n3) at (162:1 cm)  {};
  \node (n4) at (234:1 cm)  {};
  \node (n5) at (306:1 cm)  {};

  \foreach \from/\to in {n2/n3,n2/n4,n3/n5,n4/n5}
    \draw (\from) -- (\to);

\end{tikzpicture}
\caption*{$G_6$} }
\parbox{0.7in}{
\centering
\begin{tikzpicture}
  [scale=0.6,auto=left,every node/.style={circle, draw, fill=black!50,inner sep=0pt, minimum width=4pt}]
  \node (n1) at (18:1 cm) {};
  \node (n2) at (90:1 cm)  {};
  \node (n3) at (162:1 cm)  {};
  \node (n4) at (234:1 cm)  {};
  \node (n5) at (306:1 cm)  {};

  \foreach \from/\to in {n2/n3,n2/n4,n2/n5,n3/n4,n3/n5,n4/n5}
    \draw (\from) -- (\to);

\end{tikzpicture}
\caption*{$G_7$} }
\parbox{0.7in}{
\centering
\begin{tikzpicture}
  [scale=0.6,auto=left,every node/.style={circle, draw, fill=black!50,inner sep=0pt, minimum width=4pt}]
  \node (n1) at (18:1 cm) {};
  \node (n2) at (90:1 cm)  {};
  \node (n3) at (162:1 cm)  {};
  \node (n4) at (234:1 cm)  {};
  \node (n5) at (306:1 cm)  {};

  \foreach \from/\to in {n1/n5,n2/n5,n3/n4}
    \draw (\from) -- (\to);

\end{tikzpicture}
\caption*{$G_8$} }
\parbox{0.7in}{
\centering
\begin{tikzpicture}
  [scale=0.6,auto=left,every node/.style={circle, draw, fill=black!50,inner sep=0pt, minimum width=4pt}]
  \node (n1) at (18:1 cm) {};
  \node (n2) at (90:1 cm)  {};
  \node (n3) at (162:1 cm)  {};
  \node (n4) at (234:1 cm)  {};
  \node (n5) at (306:1 cm)  {};

  \foreach \from/\to in {n1/n5,n2/n5,n3/n4,n4/n5}
    \draw (\from) -- (\to);

\end{tikzpicture}
\caption*{$G_9$} }
\parbox{0.7in}{
\centering
\begin{tikzpicture}
  [scale=0.6,auto=left,every node/.style={circle, draw, fill=black!50,inner sep=0pt, minimum width=4pt}]
  \node (n1) at (18:1 cm) {};
  \node (n2) at (90:1 cm)  {};
  \node (n3) at (162:1 cm)  {};
  \node (n4) at (234:1 cm)  {};
  \node (n5) at (306:1 cm)  {};

  \foreach \from/\to in {n1/n2,n1/n4,n2/n5,n3/n4,n3/n5,n4/n5}
    \draw (\from) -- (\to);

\end{tikzpicture}
\caption*{$G_{10}$} }
\parbox{0.7in}{
\centering
\begin{tikzpicture}
  [scale=0.6,auto=left,every node/.style={circle, draw, fill=black!50,inner sep=0pt, minimum width=4pt}]
  \node (n1) at (18:1 cm) {};
  \node (n2) at (90:1 cm)  {};
  \node (n3) at (162:1 cm)  {};
  \node (n4) at (234:1 cm)  {};
  \node (n5) at (306:1 cm)  {};

  \foreach \from/\to in {n1/n5,n2/n4,n2/n5,n3/n4,n3/n5}
    \draw (\from) -- (\to);

\end{tikzpicture}
\caption*{$G_{11}$} }
\parbox{0.7in}{
\centering
\begin{tikzpicture}
  [scale=0.6,auto=left,every node/.style={circle, draw, fill=black!50,inner sep=0pt, minimum width=4pt}]
  \node (n1) at (18:1 cm) {};
  \node (n2) at (90:1 cm)  {};
  \node (n3) at (162:1 cm)  {};
  \node (n4) at (234:1 cm)  {};
  \node (n5) at (306:1 cm)  {};

  \foreach \from/\to in {n1/n5,n2/n4,n2/n5,n3/n4,n3/n5,n4/n5}
    \draw (\from) -- (\to);

\end{tikzpicture}
\caption*{$G_{12}$} }
\parbox{0.7in}{
\centering
\begin{tikzpicture}
  [scale=0.6,auto=left,every node/.style={circle, draw, fill=black!50,inner sep=0pt, minimum width=4pt}]
  \node (n1) at (18:1 cm) {};
  \node (n2) at (90:1 cm)  {};
  \node (n3) at (162:1 cm)  {};
  \node (n4) at (234:1 cm)  {};
  \node (n5) at (306:1 cm)  {};

  \foreach \from/\to in {n1/n5,n2/n3,n2/n4,n3/n4,n3/n5,n4/n5}
    \draw (\from) -- (\to);

\end{tikzpicture}
\caption*{$G_{13}$} }
\parbox{0.7in}{
\centering
\begin{tikzpicture}
  [scale=0.6,auto=left,every node/.style={circle, draw, fill=black!50,inner sep=0pt, minimum width=4pt}]
  \node (n1) at (18:1 cm) {};
  \node (n2) at (90:1 cm)  {};
  \node (n3) at (162:1 cm)  {};
  \node (n4) at (234:1 cm)  {};
  \node (n5) at (306:1 cm)  {};

  \foreach \from/\to in {n1/n5,n2/n3,n2/n4,n2/n5,n3/n4,n3/n5,n4/n5}
    \draw (\from) -- (\to);

\end{tikzpicture}
\caption*{$G_{14}$} }
\parbox{0.7in}{
\centering
\begin{tikzpicture}
  [scale=0.6,auto=left,every node/.style={circle, draw, fill=black!50,inner sep=0pt, minimum width=4pt}]
  \node (n1) at (18:1 cm) {};
  \node (n2) at (90:1 cm)  {};
  \node (n3) at (162:1 cm)  {};
  \node (n4) at (234:1 cm)  {};
  \node (n5) at (306:1 cm)  {};

  \foreach \from/\to in {n1/n4,n2/n5,n3/n4,n3/n5}
    \draw (\from) -- (\to);

\end{tikzpicture}
\caption*{$G_{15}$} }
\parbox{0.7in}{
\centering
\begin{tikzpicture}
  [scale=0.6,auto=left,every node/.style={circle, draw, fill=black!50,inner sep=0pt, minimum width=4pt}]
  \node (n1) at (18:1 cm) {};
  \node (n2) at (90:1 cm)  {};
  \node (n3) at (162:1 cm)  {};
  \node (n4) at (234:1 cm)  {};
  \node (n5) at (306:1 cm)  {};

  \foreach \from/\to in {n1/n4,n2/n5,n3/n4,n3/n5,n4/n5}
    \draw (\from) -- (\to);

\end{tikzpicture}
\caption*{$G_{16}$} }
\parbox{0.7in}{
\centering
\begin{tikzpicture}
  [scale=0.6,auto=left,every node/.style={circle, draw, fill=black!50,inner sep=0pt, minimum width=4pt}]
  \node (n1) at (18:1 cm) {};
  \node (n2) at (90:1 cm)  {};
  \node (n3) at (162:1 cm)  {};
  \node (n4) at (234:1 cm)  {};
  \node (n5) at (306:1 cm)  {};

  \foreach \from/\to in {n1/n4,n1/n5,n2/n4,n2/n5,n3/n4,n3/n5}
    \draw (\from) -- (\to);

\end{tikzpicture}
\caption*{$G_{17}$} }
\parbox{0.7in}{
\centering
\begin{tikzpicture}
  [scale=0.6,auto=left,every node/.style={circle, draw, fill=black!50,inner sep=0pt, minimum width=4pt}]
  \node (n1) at (18:1 cm) {};
  \node (n2) at (90:1 cm)  {};
  \node (n3) at (162:1 cm)  {};
  \node (n4) at (234:1 cm)  {};
  \node (n5) at (306:1 cm)  {};

  \foreach \from/\to in {n1/n4,n1/n5,n2/n4,n2/n5,n3/n4,n3/n5,n4/n5}
    \draw (\from) -- (\to);

\end{tikzpicture}
\caption*{$G_{18}$} }
\parbox{0.7in}{
\centering
\begin{tikzpicture}
  [scale=0.6,auto=left,every node/.style={circle, draw, fill=black!50,inner sep=0pt, minimum width=4pt}]
  \node (n1) at (18:1 cm) {};
  \node (n2) at (90:1 cm)  {};
  \node (n3) at (162:1 cm)  {};
  \node (n4) at (234:1 cm)  {};
  \node (n5) at (306:1 cm)  {};

  \foreach \from/\to in {n1/n4,n1/n5,n2/n3,n2/n4,n2/n5,n3/n4,n3/n5,n4/n5}
    \draw (\from) -- (\to);

\end{tikzpicture}
\caption*{$G_{19}$} }
\parbox{0.7in}{
\centering
\begin{tikzpicture}
  [scale=0.6,auto=left,every node/.style={circle, draw, fill=black!50,inner sep=0pt, minimum width=4pt}]
  \node (n1) at (18:1 cm) {};
  \node (n2) at (90:1 cm)  {};
  \node (n3) at (162:1 cm)  {};
  \node (n4) at (234:1 cm)  {};
  \node (n5) at (306:1 cm)  {};

  \foreach \from/\to in {n1/n3,n1/n5,n2/n4,n2/n5,n3/n4,n3/n5,n4/n5}
    \draw (\from) -- (\to);

\end{tikzpicture}
\caption*{$G_{20}$} }
\parbox{0.7in}{
\centering
\begin{tikzpicture}
  [scale=0.6,auto=left,every node/.style={circle, draw, fill=black!50,inner sep=0pt, minimum width=4pt}]
  \node (n1) at (18:1 cm) {};
  \node (n2) at (90:1 cm)  {};
  \node (n3) at (162:1 cm)  {};
  \node (n4) at (234:1 cm)  {};
  \node (n5) at (306:1 cm)  {};

  \foreach \from/\to in {n1/n3,n1/n4,n1/n5,n2/n3,n2/n4,n2/n5,n3/n4,n3/n5,n4/n5}
    \draw (\from) -- (\to);

\end{tikzpicture}
\caption*{$G_{21}$} }
\parbox{0.7in}{
\centering
\begin{tikzpicture}
  [scale=0.6,auto=left,every node/.style={circle, draw, fill=black!50,inner sep=0pt, minimum width=4pt}]
  \node (n1) at (18:1 cm) {};
  \node (n2) at (90:1 cm)  {};
  \node (n3) at (162:1 cm)  {};
  \node (n4) at (234:1 cm)  {};
  \node (n5) at (306:1 cm)  {};

  \foreach \from/\to in {n1/n2,n3/n4,n3/n5,n4/n5}
    \draw (\from) -- (\to);

\end{tikzpicture}
\caption*{$G_{22}$} }
\parbox{0.7in}{
\centering
\begin{tikzpicture}
  [scale=0.6,auto=left,every node/.style={circle, draw, fill=black!50,inner sep=0pt, minimum width=4pt}]
  \node (n1) at (18:1 cm) {};
  \node (n2) at (90:1 cm)  {};
  \node (n3) at (162:1 cm)  {};
  \node (n4) at (234:1 cm)  {};
  \node (n5) at (306:1 cm)  {};

  \foreach \from/\to in {n1/n2,n2/n5,n3/n4,n3/n5,n4/n5}
    \draw (\from) -- (\to);

\end{tikzpicture}
\caption*{$G_{23}$} }
\parbox{0.7in}{
\centering
\begin{tikzpicture}
  [scale=0.6,auto=left,every node/.style={circle, draw, fill=black!50,inner sep=0pt, minimum width=4pt}]
  \node (n1) at (18:1 cm) {};
  \node (n2) at (90:1 cm)  {};
  \node (n3) at (162:1 cm)  {};
  \node (n4) at (234:1 cm)  {};
  \node (n5) at (306:1 cm)  {};

  \foreach \from/\to in {n1/n2,n1/n5,n2/n5,n3/n4,n3/n5,n4/n5}
    \draw (\from) -- (\to);

\end{tikzpicture}
\caption*{$G_{24}$} }
\parbox{0.7in}{
\centering
\begin{tikzpicture}
  [scale=0.6,auto=left,every node/.style={circle, draw, fill=black!50,inner sep=0pt, minimum width=4pt}]
  \node (n1) at (18:1 cm) {};
  \node (n2) at (90:1 cm)  {};
  \node (n3) at (162:1 cm)  {};
  \node (n4) at (234:1 cm)  {};
  \node (n5) at (306:1 cm)  {};

  \foreach \from/\to in {n1/n2,n1/n4,n2/n5,n3/n4,n3/n5,n4/n5}
    \draw (\from) -- (\to);

\end{tikzpicture}
\caption*{$G_{25}$} }
\parbox{0.7in}{
\centering
\begin{tikzpicture}
  [scale=0.6,auto=left,every node/.style={circle, draw, fill=black!50,inner sep=0pt, minimum width=4pt}]
  \node (n1) at (18:1 cm) {};
  \node (n2) at (90:1 cm)  {};
  \node (n3) at (162:1 cm)  {};
  \node (n4) at (234:1 cm)  {};
  \node (n5) at (306:1 cm)  {};

  \foreach \from/\to in {n1/n2,n1/n3,n2/n4,n3/n5,n4/n5}
    \draw (\from) -- (\to);

\end{tikzpicture}
\caption*{$G_{26}$} }
\parbox{0.7in}{
\centering
\begin{tikzpicture}
  [scale=0.6,auto=left,every node/.style={circle, draw, fill=black!50,inner sep=0pt, minimum width=4pt}]
  \node (n1) at (18:1 cm) {};
  \node (n2) at (90:1 cm)  {};
  \node (n3) at (162:1 cm)  {};
  \node (n4) at (234:1 cm)  {};
  \node (n5) at (306:1 cm)  {};

  \foreach \from/\to in {n1/n2,n1/n3,n2/n4,n2/n5,n3/n4,n3/n5,n4/n5}
    \draw (\from) -- (\to);

\end{tikzpicture}
\caption*{$G_{27}$} }
\parbox{0.7in}{
\centering
\begin{tikzpicture}
  [scale=0.6,auto=left,every node/.style={circle, draw, fill=black!50,inner sep=0pt, minimum width=4pt}]
  \node (n1) at (18:1 cm) {};
  \node (n2) at (90:1 cm)  {};
  \node (n3) at (162:1 cm)  {};
  \node (n4) at (234:1 cm)  {};
  \node (n5) at (306:1 cm)  {};

  \foreach \from/\to in {n1/n2,n1/n3,n1/n5,n2/n4,n2/n5,n3/n4,n3/n5,n4/n5}
    \draw (\from) -- (\to);

\end{tikzpicture}
\caption*{$G_{28}$} }
\parbox{0.7in}{
\centering
\begin{tikzpicture}
  [scale=0.6,auto=left,every node/.style={circle, draw, fill=black!50,inner sep=0pt, minimum width=4pt}]
  \node (n1) at (18:1 cm) {};
  \node (n2) at (90:1 cm)  {};
  \node (n3) at (162:1 cm)  {};
  \node (n4) at (234:1 cm)  {};
  \node (n5) at (306:1 cm)  {};

  \foreach \from/\to in {n1/n2,n1/n3,n1/n4,n1/n5,n2/n3,n2/n4,n2/n5,n3/n4,n3/n5,n4/n5}
    \draw (\from) -- (\to);

\end{tikzpicture}
\caption*{$G_{29}$} }
\caption{Graphs of size $5$ with independence number at most $3$.}
\end{figure}

\begin{figure}[H]
\centering
\parbox{1in}{
\centering
\begin{tikzpicture}
  [scale=0.6,auto=left,every node/.style={circle, draw, fill=black!50,inner sep=0pt, minimum width=4pt}]
  \node (n3) at (60:1 cm) [label=right:$3$]{};
  \node (n1) at (180:1 cm) [label=left:$1$]{};
  \node (n2) at (300:1 cm) [label=right:$2$]{};

  \foreach \from/\to in {n1/n2}
    \draw (\from) -- (\to);

\end{tikzpicture}
\caption*{$\tau_1$} }
\parbox{0.8in}{
\centering
\begin{tikzpicture}
  [scale=0.6,auto=left,every node/.style={circle, draw, fill=black!50,inner sep=0pt, minimum width=4pt}]
  \node (n4) at (45:1 cm) {};
  \node (n2) at (135:1 cm) [label=left:$2$]{};
  \node (n1) at (225:1 cm) [label=left:$1$]{};
  \node (n3) at (315:1 cm) [label=right:$3$]{};

  \foreach \from/\to in {n1/n2,n2/n4}
    \draw (\from) -- (\to);

\end{tikzpicture}
\caption*{$M_1$} }
\parbox{0.8in}{
\centering
\begin{tikzpicture}
  [scale=0.6,auto=left,every node/.style={circle, draw, fill=black!50,inner sep=0pt, minimum width=4pt}]
  \node (n4) at (45:1 cm) {};
  \node (n2) at (135:1 cm) [label=left:$2$]{};
  \node (n1) at (225:1 cm) [label=left:$1$]{};
  \node (n3) at (315:1 cm) [label=right:$3$]{};

  \foreach \from/\to in {n1/n2,n1/n4}
    \draw (\from) -- (\to);

\end{tikzpicture}
\caption*{$M_2$} }
\parbox{0.8in}{
\centering
\begin{tikzpicture}
  [scale=0.6,auto=left,every node/.style={circle, draw, fill=black!50,inner sep=0pt, minimum width=4pt}]
  \node (n4) at (45:1 cm) {};
  \node (n2) at (135:1 cm) [label=left:$2$]{};
  \node (n1) at (225:1 cm) [label=left:$1$]{};
  \node (n3) at (315:1 cm) [label=right:$3$]{};

  \foreach \from/\to in {n1/n2,n2/n4,n3/n4}
    \draw (\from) -- (\to);

\end{tikzpicture}
\caption*{$M_3$} }
\parbox{0.8in}{
\centering
\begin{tikzpicture}
  [scale=0.6,auto=left,every node/.style={circle, draw, fill=black!50,inner sep=0pt, minimum width=4pt}]
  \node (n4) at (45:1 cm) {};
  \node (n2) at (135:1 cm) [label=left:$2$]{};
  \node (n1) at (225:1 cm) [label=left:$1$]{};
  \node (n3) at (315:1 cm) [label=right:$3$]{};

  \foreach \from/\to in {n1/n2,n1/n4,n3/n4}
    \draw (\from) -- (\to);

\end{tikzpicture}
\caption*{$M_4$} }
\caption{Type $\tau_1$ and its flags of size $4$.}
\end{figure}

\begin{figure}[H]
\centering
\parbox{1in}{
\centering
\begin{tikzpicture}
  [scale=0.6,auto=left,every node/.style={circle, draw, fill=black!50,inner sep=0pt, minimum width=4pt}]
  \node (n3) at (60:1 cm) [label=right:$3$]{};
  \node (n1) at (180:1 cm) [label=left:$1$]{};
  \node (n2) at (300:1 cm) [label=right:$2$]{};

  \foreach \from/\to in {n1/n2,n1/n3}
    \draw (\from) -- (\to);

\end{tikzpicture}
\caption*{$\tau_2$}}
\parbox{0.8in}{
\centering
\begin{tikzpicture}
  [scale=0.6,auto=left,every node/.style={circle, draw, fill=black!50,inner sep=0pt, minimum width=4pt}]
  \node (n4) at (45:1 cm) {};
  \node (n2) at (135:1 cm) [label=left:$2$]{};
  \node (n1) at (225:1 cm) [label=left:$1$]{};
  \node (n3) at (315:1 cm) [label=right:$3$]{};

  \foreach \from/\to in {n1/n2,n1/n3}
    \draw (\from) -- (\to);

\end{tikzpicture}
\caption*{$N_1$} }
\parbox{0.8in}{
\centering
\begin{tikzpicture}
  [scale=0.6,auto=left,every node/.style={circle, draw, fill=black!50,inner sep=0pt, minimum width=4pt}]
  \node (n4) at (45:1 cm) {};
  \node (n2) at (135:1 cm) [label=left:$2$]{};
  \node (n1) at (225:1 cm) [label=left:$1$]{};
  \node (n3) at (315:1 cm) [label=right:$3$]{};

  \foreach \from/\to in {n1/n2,n1/n3,n1/n4}
    \draw (\from) -- (\to);

\end{tikzpicture}
\caption*{$N_2$} }
\parbox{0.8in}{
\centering
\begin{tikzpicture}
  [scale=0.6,auto=left,every node/.style={circle, draw, fill=black!50,inner sep=0pt, minimum width=4pt}]
  \node (n4) at (45:1 cm) {};
  \node (n2) at (135:1 cm) [label=left:$2$]{};
  \node (n1) at (225:1 cm) [label=left:$1$]{};
  \node (n3) at (315:1 cm) [label=right:$3$]{};

  \foreach \from/\to in {n1/n2,n1/n3,n1/n4,n3/n4}
    \draw (\from) -- (\to);

\end{tikzpicture}
\caption*{$N_3$} }
\parbox{0.8in}{
\centering
\begin{tikzpicture}
  [scale=0.6,auto=left,every node/.style={circle, draw, fill=black!50,inner sep=0pt, minimum width=4pt}]
  \node (n4) at (45:1 cm) {};
  \node (n2) at (135:1 cm) [label=left:$2$]{};
  \node (n1) at (225:1 cm) [label=left:$1$]{};
  \node (n3) at (315:1 cm) [label=right:$3$]{};

  \foreach \from/\to in {n1/n2,n1/n3,n1/n4,n2/n4}
    \draw (\from) -- (\to);

\end{tikzpicture}
\caption*{$N_4$} }
\parbox{0.8in}{
\centering
\begin{tikzpicture}
  [scale=0.6,auto=left,every node/.style={circle, draw, fill=black!50,inner sep=0pt, minimum width=4pt}]
  \node (n4) at (45:1 cm) {};
  \node (n2) at (135:1 cm) [label=left:$2$]{};
  \node (n1) at (225:1 cm) [label=left:$1$]{};
  \node (n3) at (315:1 cm) [label=right:$3$]{};

  \foreach \from/\to in {n1/n2,n1/n3,n3/n4}
    \draw (\from) -- (\to);

\end{tikzpicture}
\caption*{$N_5$} }
\parbox{0.8in}{
\centering
\begin{tikzpicture}
  [scale=0.6,auto=left,every node/.style={circle, draw, fill=black!50,inner sep=0pt, minimum width=4pt}]
  \node (n4) at (45:1 cm) {};
  \node (n2) at (135:1 cm) [label=left:$2$]{};
  \node (n1) at (225:1 cm) [label=left:$1$]{};
  \node (n3) at (315:1 cm) [label=right:$3$]{};

  \foreach \from/\to in {n1/n2,n1/n3,n2/n4}
    \draw (\from) -- (\to);

\end{tikzpicture}
\caption*{$N_6$} }
\parbox{0.8in}{
\centering
\begin{tikzpicture}
  [scale=0.6,auto=left,every node/.style={circle, draw, fill=black!50,inner sep=0pt, minimum width=4pt}]
  \node (n2) at (135:1 cm) [label=left:$2$]{};
  \node (n4) at (45:1 cm) {};
  \node (n1) at (225:1 cm) [label=left:$1$]{};
  \node (n3) at (315:1 cm) [label=right:$3$]{};

  \foreach \from/\to in {n1/n2,n1/n3,n1/n4,n2/n4,n3/n4}
    \draw (\from) -- (\to);

\end{tikzpicture}
\caption*{$N_7$} }
\parbox{0.8in}{
\centering
\begin{tikzpicture}
  [scale=0.6,auto=left,every node/.style={circle, draw, fill=black!50,inner sep=0pt, minimum width=4pt}]
  \node (n4) at (45:1 cm) {};
  \node (n2) at (135:1 cm) [label=left:$2$]{};
  \node (n1) at (225:1 cm) [label=left:$1$]{};
  \node (n3) at (315:1 cm) [label=right:$3$]{};

  \foreach \from/\to in {n1/n2,n1/n3,n2/n4,n3/n4}
    \draw (\from) -- (\to);

\end{tikzpicture}
\caption*{$N_8$} } \caption{Type $\tau_2$ and its flags of size $4$.}
\end{figure}

\begin{figure}[H]
\centering
\parbox{0.8in}{
\centering
\begin{tikzpicture}
  [scale=0.6,auto=left,every node/.style={circle, draw, fill=black!50,inner sep=0pt, minimum width=4pt}]
  \node (n1) at (0:0 cm) [label=left:$1$]{};

  \foreach \from/\to in {}
    \draw (\from) -- (\to);

\end{tikzpicture}
\caption*{dot}}
\parbox{0.8in}{
\centering
\begin{tikzpicture}
  [scale=0.6,auto=left,every node/.style={circle, draw, fill=black!50,inner sep=0pt, minimum width=4pt}]
  \node (n2) at (0:1 cm) {};
  \node (n1) at (180:1 cm) [label=left:$1$]{};

  \foreach \from/\to in {n1/n2}
    \draw (\from) -- (\to);

\end{tikzpicture}
\caption*{$\rho$} }
\parbox{0.8in}{
\centering
\begin{tikzpicture}
  [scale=0.6,auto=left,every node/.style={circle, draw, fill=black!50,inner sep=0pt, minimum width=4pt}]
  \node (n2) at (60:1 cm) {};
  \node (n1) at (180:1 cm) [label=left:$1$]{};
  \node (n3) at (300:1 cm) {};

  \foreach \from/\to in {}
    \draw (\from) -- (\to);

\end{tikzpicture}
\caption*{$Z_1$} }
\parbox{0.8in}{
\centering
\begin{tikzpicture}
  [scale=0.6,auto=left,every node/.style={circle, draw, fill=black!50,inner sep=0pt, minimum width=4pt}]
  \node (n2) at (60:1 cm) {};
  \node (n1) at (180:1 cm) [label=left:$1$]{};
  \node (n3) at (300:1 cm) {};

  \foreach \from/\to in {n1/n3}
    \draw (\from) -- (\to);

\end{tikzpicture}
\caption*{$Z_2$} }
\parbox{0.8in}{
\centering
\begin{tikzpicture}
  [scale=0.6,auto=left,every node/.style={circle, draw, fill=black!50,inner sep=0pt, minimum width=4pt}]
  \node (n2) at (60:1 cm) {};
  \node (n1) at (180:1 cm) [label=left:$1$]{};
  \node (n3) at (300:1 cm) {};

  \foreach \from/\to in {n2/n3}
    \draw (\from) -- (\to);

\end{tikzpicture}
\caption*{$Z_3$} }
\parbox{0.8in}{
\centering
\begin{tikzpicture}
  [scale=0.6,auto=left,every node/.style={circle, draw, fill=black!50,inner sep=0pt, minimum width=4pt}]
  \node (n2) at (60:1 cm) {};
  \node (n1) at (180:1 cm) [label=left:$1$]{};
  \node (n3) at (300:1 cm) {};

  \foreach \from/\to in {n1/n3,n2/n3}
    \draw (\from) -- (\to);

\end{tikzpicture}
\caption*{$Z_4$} }
\parbox{0.8in}{
\centering
\begin{tikzpicture}
  [scale=0.6,auto=left,every node/.style={circle, draw, fill=black!50,inner sep=0pt, minimum width=4pt}]
  \node (n2) at (60:1 cm) {};
  \node (n1) at (180:1 cm) [label=left:$1$]{};
  \node (n3) at (300:1 cm) {};

  \foreach \from/\to in {n1/n2,n2/n3,n1/n3}
    \draw (\from) -- (\to);

\end{tikzpicture}
\caption*{$Z_5$} } \caption{Type \emph{dot} and its flags.}
\end{figure}

In the subsequent lemmas, for each type used in the proof, we express the corresponding positive semi-definite matrices as squares of flags, and give their expansions into graphs of size 5.  The coefficients were obtained through the use of a computer program, but can easily be verified by hand, just as in the previous section.

\begin{lemma}
For the type $\tau_1$, we have
\begin{multline*}\begin{split}
\Delta_1 &= [[(M_1-M_2)^2]]_{\tau_1}\\
         &=\frac{1}{30}\big(G_{2} - G_{3} -4 G_{6}\big),
\end{split}\end{multline*}
\vspace{-0.2in}
\begin{multline*}\begin{split}
\Delta_2 &= [[(3M_1 -3M_2 -10M_3 +10M_4)^2]]_{\tau_1}\\
&=\frac{1}{30}\big(9 G_{2} -9 G_{3} -36 G_{6} -60 G_{9} +160 G_{11} +100 G_{13}+60 G_{15} -60 G_{16} -100 G_{25} -500 G_{26}\big).
\end{split}\end{multline*}
\end{lemma}

\medskip

\begin{lemma}
For the type $\tau_2$, we have
\begin{multline*}\begin{split}
\Delta_3 =& [[(-3N_1-N_2+3N_3+3N_4)^2]]_{\tau_2}\\
=&\frac{1}{30}\big(-18 G_{2} +9 G_{8} +3 G_{9} -11 G_{10} +3 G_{12} +27 G_{14} -18 G_{16} +9 G_{20} +36 G_{24}\big),
\end{split}\end{multline*}
\vspace{-0.2in}
\begin{multline*}\begin{split}
\Delta_4 =& [[(-20N_1 -20N_2 +11N_3 +11N_4 + 9N_5+9N_6)^2]]_{\tau_1}\\
=&\frac{1}{30}\big(-440 G_{2} -360 G_{3} +400 G_{8} +121 G_{9} -480 G_{10} -360 G_{11}-319 G_{12} +198 G_{13} +363 G_{14} \\
&-279 G_{15} -242 G_{16} +121 G_{20} +279 G_{23} +484 G_{24} +198 G_{25} +405 G_{26}\big),
\end{split}\end{multline*}
\vspace{-0.2in}
\begin{multline*}\begin{split}
\Delta_5 =& [[(-19N_1 -15N_2 +15N_3 +15N_4 + 4N_5+4N_6+15N_7)^2]]_{\tau_1}\\
=&\frac{1}{30}\big(-570 G_{2} -152 G_{3} -570 G_{4} +361 G_{8} +181 G_{9}-675 G_{10}-120 G_{11} -735 G_{12} +240 G_{13} \\
&+675 G_{14} -136 G_{15} -450 G_{16} -1350 G_{18} +900 G_{19} +795 G_{20} +675 G_{21} +136 G_{23}+900 G_{24} \\
&+120 G_{25} +80 G_{26} +450 G_{28}\big),
\end{split}\end{multline*}
\vspace{-0.2in}
\begin{multline*}\begin{split}
\Delta_6 =& [[(-6N_1 -14N_2 -2N_3 -2N_4 + 8N_5+8N_6-5N_7+10N_8)^2]]_{\tau_1}\\
=&\frac{1}{30}\big(+24 G_{2} -96 G_{3} +60 G_{4} -240 G_{6} +36 G_{8}-76 G_{9} +308 G_{10} -264 G_{11} +160 G_{12} -112 G_{13} \\
&+12 G_{14} -32 G_{15} -8 G_{16} -540 G_{17} +420 G_{18}+40 G_{19}-56 G_{20} +75 G_{21} +32 G_{23} +16 G_{24}\\
&+88 G_{25} +320 G_{26} -80 G_{27} -150 G_{28}\big).
\end{split}\end{multline*}
\end{lemma}

\medskip

\begin{lemma}
For the type \emph{dot}, we have
\begin{multline*}\begin{split}
\Delta_7 =& [[(-2Z_1 + Z_2)^2]]_{dot}\\
=&\frac{1}{15} \big(6 G_{1} + 2G_{2} +2 G_{3} -8 G_{5} +4 G_{6} -10 G_{8} -4 G_{9} -2 G_{11}-G_{15} +G_{16} +6 G_{22} +2 G_{23} \\
&+ G_{25} +5 G_{26}\big),
\end{split}\end{multline*}
\vspace{-0.2in}
\begin{multline*}\begin{split}
\Delta_8 =& [[(-2Z_1 -Z_2 + 4Z_3)^2]]_{dot}\\
=&\frac{1}{15} \big(-42 G_{1} -2 G_{2} +10 G_{3} +24 G_{4} +24 G_{5} +36 G_{6} +48 G_{7} -2 G_{8} -4 G_{9}+2 G_{11} -4 G_{13}\\
& - G_{15} -7 G_{16} -18 G_{22} -6 G_{23} + G_{25} +5 G_{26}\big),
\end{split}\end{multline*}
\vspace{-0.2in}
\begin{multline*}\begin{split}
\Delta_9 =& [[(7Z_1 -4Z_2 + Z_3 + 3Z_4)^2]]_{dot}\\
=&\frac{1}{15}\big(138 G_{1} +61 G_{2} +43 G_{3} -39 G_{4} -141 G_{5} -45 G_{6} +3 G_{7} -146 G_{8}-19 G_{9}+42 G_{10}\\
&-52 G_{11} +21 G_{12} -22 G_{13} +9 G_{14} -25 G_{15}+65 G_{16} +54 G_{17} +54 G_{18} +18 G_{19} -6 G_{20}\\
&+72 G_{22} -21 G_{23} -96 G_{24} +19 G_{25} +125 G_{26} +18 G_{27}\big),
\end{split}\end{multline*}
\vspace{-0.2in}
\begin{multline*}\begin{split}
\Delta_{10} =& [[(8Z_1 -2Z_2-9Z_3 + 10Z_5)^2]]_{dot}\\
=&\frac{1}{15} \big(-168 G_{1} +103 G_{2} +85 G_{3} +170 G_{4} +153 G_{5} +226 G_{6} +3 G_{7} -16 G_{8}+4 G_{9} +160 G_{10}\\
&-16 G_{11} +120 G_{12} -132 G_{13} -120 G_{14} +16 G_{15} +80 G_{16} +240 G_{18} +70 G_{19} -120 G_{20}\\
&+600 G_{21} -138 G_{22} -136 G_{23} -260 G_{24} -86 G_{25} +20 G_{26} +200 G_{28} +1500 G_{29}\big),
\end{split}\end{multline*}
\vspace{-0.2in}
\begin{multline*}\begin{split}
\Delta_{11} =& \left[\left[\left(\rho-\frac{1}{3}\right)^2\right]\right]_{dot}\\
=&\frac{1}{90}\big(G_{1}+G_{2}-2G_{3}+4G_{4}-2G_{5}-2G_{6}+10G_{7}-5G_{8}-2G_{9}+4G_{10}-2G_{11}+7G_{12}\\
&+4G_{13}+13G_{14}-5G_{15}+G_{16}+G_{17}+13G_{18}+19G_{19}+10G_{20}+28G_{21}-5G_{22}-2G_{23}\\
&+4G_{24}+G_{25}-5G_{26}+7G_{27}+16G_{28}+40G_{29}\big).
\end{split}\end{multline*}
\end{lemma}

\medskip

We can now combine these lemmas to obtain an asymptotic lower bound on the density of triangles, $K_3$, in any $\overline{K_4}$-free graph.

\begin{thm} \label{thm34flag}
We have
\[
K_3 - \sum_{i=1}^{11} c_i \Delta_i \ge \frac{1}{9} \sum_{j=1}^{29} G_j = \frac{1}{9},
\]
where
\[
\mathbf{c}=(c_i)_{i=1}^{11} = \frac{1}{2^{5}\cdot 3 \cdot 1009}\left(263984, 4720, 4432,\frac{412192}{371},\frac{72789}{112},\frac{4655105}{3392},1185, 8437, 3440, 856,1128\right).
\]
\end{thm}
\begin{proof}
We begin by expanding $K_3$ into graphs of size $5$.  A straightforward calculation gives
\begin{align*}
    K_3 = \frac{1}{10} \bigg( G_1 + G_2 &+ 2 G_4 + 4 G_7 + G_{10} + 2 G_{12} + 2 G_{13} + 4 G_{14} + G_{16} + 3 G_{18} + 5 G_{19}\\
    & + 3 G_{20} + 7 G_{21} + G_{22} + G_{23} + 2 G_{24} + G_{25} + 2 G_{27} + 4 G_{28} + 10 G_{29} \bigg).
\end{align*}
We now use the lemmas to expand the squares $\Delta_i$ into the graphs $G_j$.  After summing the coefficients in the linear combination, it can easily be verified that they are all at least $\frac{1}{9}$.  Since the densities must sum to $1$, we have $\sum_{j=1}^{29} G_j = 1$, which gives the final equality.
\end{proof}

\begin{cor} \label{cor34flag}
Any $n$-vertex graph $G$ with $\alpha(G) \le 3$ satisfies
\[ \frac{t_3(G)}{ \binom{n}{3} } - \frac{47}{4036n} \sum_{v} \left( \frac{d(v)}{n-1} - \frac{1}{3} \right)^2 \ge \frac{1}{9} - o_{n \rightarrow \infty}(1). \]
\end{cor}
\begin{proof}
Since the $\Delta_i$ are squares of flags, they are asymptotically non-negative.  Hence discarding the terms for $\Delta_i$, $1 \le i \le 10$, maintains the inequality.  This gives $K_3 - \frac{47}{4036} \left[ \left[ \left( \rho - \frac{1}{3} \right)^2 \right]\right]_{dot} \ge \frac{1}{9} - o_{n \rightarrow \infty}(1)$.  Interpreting these terms combinatorially gives the corollary.
\end{proof}

\subsection{The stability analysis}

In order to derive a stability result for the $(3,4)$-problem, we use the following well-known result of Andr\'{a}sfai, Erd\H{o}s and S\'{o}s \cite{andrasfai}.

\begin{thm} \emph{(Andr\'{a}sfai, Erd\"{o}s, S\'{o}s)} \label{aes}
A $K_r$-free graph on $n$ vertices that has minimum degree larger
than $\frac{3r-7}{3r-4} n$ must be $(r-1)$-partite.
\end{thm}

Applying this to the complement of a graph with $r = 4$, we find
that a graph $G$ on $n$ vertices with $\alpha(G) \le 3$ and maximum
degree less than $\frac{3}{8}n$ must be spanned by three cliques.
The following stability result follows.

\begin{prop} \label{34stability}
Suppose $0 < \varepsilon < \frac{1}{30}$.  There exists $n_0 = n_0(\varepsilon)$ such that any graph $G$ on $n \ge n_0$ vertices with $\alpha(G) \le 3$ and $t_3(G) < \left( \frac{1}{9} + \varepsilon^5 \right) \binom{n}{3}$ contains an induced subgraph $G' \subset G$ on at least $\left(1 - 100 \varepsilon^3 \right) n$ vertices that is spanned by three cliques of size between $\left( \frac{1}{3} - 3 \varepsilon \right)n $ and $\left( \frac{1}{3} + \varepsilon \right)n$.  Moreover, every vertex in $G'$ sends at most $4 \varepsilon n$ edges outside its clique.
\end{prop}

\begin{proof}
We have from Corollary \ref{cor34flag} that for any graph $G$ on $n$ vertices with $\alpha(G) \le 3$,
\[ \frac{t_3(G)}{\binom{n}{3}} - \frac{47}{4036n} \sum_v \left( \frac{d(v)}{n-1} - \frac{1}{3} \right)^2 \ge \frac{1}{9} - o_{n \rightarrow \infty}(1). \]
In particular, if $t_3(G) < \left( \frac{1}{9} + \varepsilon^5 \right) \binom{n}{3}$, and $n$ is large enough, then
\[ \sum_v \left( \frac{d(v)}{n-1} - \frac{1}{3} \right)^2 < 100 \varepsilon^5 n. \]

Let $B = \left\{ v : d(v) \ge \left( \frac{1}{3} + \varepsilon \right) n \right\}$.  Then $|B| \varepsilon^2 < \sum_v \left( \frac{d(v)}{n-1} - \frac{1}{3} \right)^2 < 100 \varepsilon^5 n$, and so $|B| < 100 \varepsilon^3 n$.

Let $G'$ be the induced subgraph on $V(G) \setminus B$.  As claimed, $G'$ has $n' \ge (1 - 100 \varepsilon^3)n$ vertices.  Moreover, since $\varepsilon < \frac{1}{30}$ the maximum degree $\Delta(G')$ is bounded by
\[ \Delta(G') < \left( \frac{1}{3} + \varepsilon \right) n \le \frac{ \frac{1}{3} + \varepsilon }{1 - 100 \varepsilon^3} n' < \frac{3}{8} n'. \]
Hence we can apply Theorem \ref{aes} in its complementary form to deduce that $G'$ is spanned by three cliques.

Since $\Delta(G') < \left( \frac{1}{3} + \varepsilon \right)n$, we
deduce that the largest clique in $G'$ has size at most $\left(
\frac{1}{3} + \varepsilon \right)n$.  This implies that the smallest
clique has size at least $\left( 1 - 100 \varepsilon^3 \right) n - 2
\left( \frac{1}{3} + \varepsilon \right) n > \left( \frac{1}{3} - 3
\varepsilon \right) n$ (using the bound $\varepsilon <
\frac{1}{30}$).  This implies that every vertex in $G'$ can send at
most $\left( \frac{1}{3} + \varepsilon \right) n - \left(
\frac{1}{3} - 3 \varepsilon \right) n = 4 \varepsilon n$ edges
outside its own clique.

Finally, consider the vertices in $B$.  If any vertex $v \in B$ is
adjacent to all vertices in one of the cliques $C_i$, and does not
have more than $4 \varepsilon n$ edges outside $C_i$, then we can
add $v$ to $C_i$ without affecting any of the previous bounds.  Thus
the only vertices left in $B$ are either those adjacent to one
clique, but with too many neighbors outside the clique, or those
with a non-neighbor in each of the three cliques.
\end{proof}

\medskip

This stability result allows us to, for large values of $n$, deduce the exact value of the $(3,4)$-problem, and also to characterise all extremal graphs.  Recall that we define $f(n,k,l)$ to be the minimum of $t_k(G)$ over all graphs $G$ on $n$ vertices with $\alpha(G) \le l-1$.

\begin{thm} \label{34exact}
There exists $n_0$ such that for every $n \ge n_0$, $f(n,3,4) = \binom{\left \lfloor n / 3 \right \rfloor}{3} + \binom{\left \lfloor (n+1) / 3 \right \rfloor}{3} + \binom{\left \lfloor (n+2) / 3 \right \rfloor}{3}$.  Moreover, if $G$ is a graph on $n \ge n_0$ vertices with $t_3(G) = f(n,3,4)$, then $G$ contains $\overline{T_{n,3}}$, a disjoint union of three nearly-equal cliques.
\end{thm}

\begin{proof} First note that $G = \overline{T_{n,3}}$ has $\alpha(G) \le 3$ and so we have the upper bound $f(n,3,4) \le t_3(\overline{T_{n,3}}) = \binom{ \left \lfloor n / 3 \right \rfloor}{3} + \binom{ \left \lfloor (n+1) / 3 \right \rfloor}{3} + \binom{ \left \lfloor (n+2) / 3 \right \rfloor}{3} \sim \frac{1}{9} \binom{n}{3}$ - note that this upper bound holds for all $n$.

To obtain a matching lower bound, we apply the stability result from Proposition \ref{34stability}.  Take $\varepsilon = \frac{1}{100}$, and let $n \ge n_0 ( \varepsilon )$ be sufficiently large.  Suppose $G$ is an extremal graph on $n \ge n_0$ vertices.  In particular, we have $t_3(G) < \left( \frac{1}{9} + \varepsilon^5 \right) \binom{n}{3}$ for $n$ large enough.   From the proof of the proposition, we know that there is a set $B$ of at most $100 \varepsilon^3 n$ `bad' vertices, and the remainder of the vertices are in three cliques, with at most $4 \varepsilon n$ edges to the other cliques.  Label the cliques in order of size, say $|C_1| \ge |C_2| \ge |C_3|$.  We will show that an extremal graph cannot have any bad vertices, so $G$ is spanned by the three cliques.  We begin with a simple observation.

\textbf{Claim:}  Every vertex $v \in V(G)$ is in at most
$\binom{|C_3| + |B| }{2}$ triangles.

\textbf{Proof:}      Suppose some vertex $v$ were in more triangles.
Delete $v$, and add a new vertex $v'$ with $N(v') = C_3 \cup B$.
This does not increase the independence number, and $v'$ is in at
most $\binom{|C_3| + |B|}{2}$ triangles.  Hence we have decreased
the number of triangles in $G$, which contradicts the minimality of
$G$. Note that $\binom{|C_3| + |B|}{2} \le \binom{|C_3|}{2} + |B|n
\le \binom{|C_3|}{2} + 100 \varepsilon^3 n^2 = \binom{|C_3|}{2} +
\varepsilon^2 n^2$.

\medskip

Now consider a potential bad vertex $v \in B$.  There are two reasons $v$ could be bad:

\medskip

\textbf{Case 1:}  $v$ is adjacent to all vertices of one of the
cliques $C_i$, but has more than $4 \varepsilon n$ neighbors in the
other cliques.

If $v$ has more than $4 \varepsilon n$ neighbors in the other
cliques, it must have at least $2 \varepsilon n$ neighbors in one of
them.  Note that every pair of these neighbors creates a triangle
with $v$.  Thus $v$ is in at least $\binom{|C_i|}{2} + \binom{ 2
\varepsilon n}{2} > \binom{|C_3|}{2} + \varepsilon^2 n^2$ triangles,
which contradicts our earlier claim.  Hence this case cannot occur.

\medskip

\textbf{Case 2:} $v$ has a non-neighbor in each of the three
cliques.

Let $\overline{d}_i = |C_i \setminus N(v)|$ be the number of non-neighbors of $v$ in the $i$th clique.  Consider the cliques in increasing order of these values, that is, suppose $\overline{d}_{i_1} \le \overline{d}_{i_2} \le \overline{d}_{i_3}$.  Let $x$ be a non-neighbor of $v$ in $C_{i_1}$.

\textbf{Case 2a:} Every vertex $y \in C_{i_2}$ is adjacent to one of
$\{ v, x \}$.

Since $x$ has at most $4 \varepsilon n$ neighbors in $C_{i_2}$, it follows that $\overline{d}_{i_1} \le \overline{d}_{i_2} \le 4 \varepsilon n$.  Counting only the neighbors of $v$ in the cliques $C_{i_1}$ and $C_{i_2}$, we see that $v$ is in at least $\binom{|C_{i_1}| - 4 \varepsilon n}{2} + \binom{|C_{i_2}| - 4 \varepsilon n}{2} \ge 2 \binom{|C_3| - 4 \varepsilon n }{2}$ triangles.  We have $2 \binom{|C_3| - 4 \varepsilon n}{2} \approx |C_3|^2 - 8 \varepsilon |C_3| n + 16 \varepsilon^2 n^2$.  Since $|C_3| \ge \left ( \frac{1}{3} - 3 \varepsilon \right) n$ and $\varepsilon = \frac{1}{100}$, this is greater than $\binom{|C_3|}{2} + \varepsilon^2 n^2 < \frac{1}{2} |C_3|^2 + \varepsilon^2 n^2$, which contradicts the earlier claim.

\medskip

\textbf{Case 2b:} $v$ and $x$ have a common non-neighbor in
$C_{i_2}$, say $y$.

In this case, as $\alpha(G) \le 3$, every vertex in $C_{i_3}$ must
be adjacent to one of $\{ v, x, y \}$.  Since $x$ and $y$ have at
most $4 \varepsilon n$ neighbors in $C_{i_3}$, it follows that
$\overline{d}_{i_1} \le \overline{d}_{i_2} \le \overline{d}_{i_3}
\le 8 \varepsilon n$.  Thus $v$ is in at least $\binom{ |C_{i_1}| -
8 \varepsilon n}{2} + \binom{ |C_{i_2}| - 8 \varepsilon n}{2} +
\binom{ |C_{i_3}| - 8 \varepsilon n}{2} \ge 3 \binom{ |C_3| - 8
\varepsilon n}{2} \approx \frac{3}{2} |C_3|^2 - 24 \varepsilon |C_3|
n + 96 \varepsilon^2 n^2$ triangles.  Again, given our bounds on
$|C_3|$ and $\varepsilon$, this is greater than $\binom{|C_3|}{2} +
\varepsilon^2 n^2$, which gives a contradiction.

\medskip

Thus we have shown that in an extremal graph, there are no bad vertices, and so the three cliques span all $n$ vertices and $|B| = 0$.  Now note that any vertex in $C_1$ is in $ \binom{ |C_1| - 1}{2}$ triangles from within $C_1$ alone.  By the earlier claim, we must have $\binom{|C_1| - 1}{2} \le \binom{|C_3| + |B|}{2} = \binom{|C_3|}{2}$, from which it follows that $|C_1| - 1\le |C_3|$.  Thus $|C_3| \le |C_2| \le |C_1| \le |C_3| + 1$, which shows that the cliques must be nearly equal in size.

\medskip

This implies that $\overline{T_{n,3}} \subset G$, and so it follows that for any graph $G$ on $n$ vertices with $\alpha(G) \le 3$, we must have $t_3(G) \ge t_3( \overline{T_{n,3}})$.  Thus $f(n,3,4) = t_3( \overline{T_{n,3}})$.  Moreover, if $G$ is an extremal graph, then since we have equality, there can be no triangles with vertices from different cliques.  This means that each vertex can have at most one neighbor in each of the two other cliques; in other words, the bipartite graphs between cliques are (partial) matchings.  These matchings must be such that there is no triangle with one vertex from each clique.  However, the extremal graph is not unique, as there are many possibilities for the matchings.

\end{proof}

\section{Concluding Remarks} \label{concluding}

In this paper, we apply the techniques of flag algebras, combined with stability arguments, to solve the Erd\H{o}s problem for the cases $(k,l) = (4,3)$ and $(3,4)$.  In particular, we show that Nikiforov's construction of a blow-up of $C_5$ is optimal for the $(4,3)$-problem, while Erd\H{o}s' conjecture still holds for the $(3,4)$-problem.

We have also run the SDP problem for larger cases, and our calculations suggests that Erd\H{o}s' conjecture remains valid for the $(3,5)$- and $(3,6)$-problems.  Moreover, it would appear that a blow-up of $C_5$ is also optimal for the $(5,3)$-problem.  Since this paper is already quite long, we decided not to process the SDP results to find rational solutions.  However, after doing so it should be possible to develop stability results similar to those above, and thus to determine the exact solution to these problems.

Note that the extremal graphs we have found are all blow-ups of small graphs.  In particular, the graphs are Ramsey graphs.  The construction of $l-1$ cliques is a blow-up of an independent set of size $l-1$, which is the $R(2,l)$ Ramsey graph.  On the other hand, $C_5$ is the $R(3,3)$ Ramsey graph.  One may therefore ask if, for large $n$, the solution of the $(k,l)$-problem is always a blow-up of an $R(s,t)$ Ramsey graph, where $s$ and $t$ depend only on $k$ and $l$.  Solving this problem in general appears to be quite difficult.

A simpler question, first asked by Nikiforov, is to determine the extremal graphs for the $(k,l)$-problem as one parameter is fixed and the other grows.  In particular, it remains to determine for which values of $l$ a disjoint union of $l-1$ cliques remains optimal for the $(3,l)$-problem.  In light of the above results, one could also study for which values of $k$ the blow-up of $C_5$ is optimal for the $(k,3)$-problem.  Proofs by flag algebras are infeasible for large values of $k$ and $l$, as the search space and running time grow exponentially in these parameters.  It would be of great interest to develop new techniques to attack this problem.

\newpage 

\appendix

\section{Implementation of flag algebras} \label{app_impflags}

In Section \ref{flag_intro}, we covered the basics of the theory
behind flag algebras; here we discuss the actual
implementation of the method.  In particular, we will discuss how
to set up the SDP problem, and then find a verifiable proof. The
main steps are:
\begin{enumerate}
\item Identifying the types $\sigma_i$ to use, and finding a suitable size $t$ for the expansion of the positive semi-definite matrices.
\item Finding a verifiable (e.g. rational) solution that leads to a proof.
\item \textit{(Optional)} Writing the positive semi-definite matrix as a sum of squares.
\end{enumerate}
We shall address each of these steps in turn.

\medskip

\noindent {\sc Identifying types:}

The process of identifying the necessary types $\sigma_i$ and
finding a suitable size $t$ essentially comes down to
trial-and-error.  Note that whatever choice of types and size we
make will result in an SDP problem as outlined above, which can then
be solved to provide \emph{some} bound for the extremal problem.  In
order to determine whether or not this is the \emph{right} bound, we
need a conjecture on what the bound should be - this typically comes
from a construction.  We then seek to keep improving the flag
algebra results until they match the conjectured bound.

To produce the flag algebra results, we start with the initial size
$t$ to be the size of the subgraph $J$, the density of which we are
trying to bound. Given $t$, we produce a list of all admissible
graphs $G$ of size $t$.  We then consider all possible types of size
suitable for expansion into graphs of size $t$. Recall that if we
have a type of size $k$, and use flags of size $l \ge k+1$, then to compute
a product of two flags, we must expand into graphs of size at least
$2l - k \ge k+2$. This restricts the size of types and flags we can use -
our types can be of size at most $t-2$, and given a type of size
$k$, we choose the largest possible size of flags $l$ that satisfies
$2l - k \le t$.

For each of our types $\sigma_i$, with its associated list of flags
$\mathcal{F}^{\sigma_i}_{l_i}$, we compute the product of each pair
of flags, which gives the corresponding block in the SDP problem.
This provides the formulation of the SDP problem, which can then be
solved numerically.

If the numerical bound is less than the conjecture, then we do not
have enough types to solve the problem.  Thus we increase the size
$t$, which allows the use of larger types, and repeat the process.
If the numerical bound matches the conjecture, we then have enough
types to solve the problem, and can proceed to finding a verifiable
proof.

\medskip

At this stage, we have the block variable matrices $Q_i$ for the SDP
problem. However, as they were computed numerically, they are
subject to rounding error, and thus we cannot be certain that they
are truly positive semi-definite matrices, nor that the bound for
the extremal problem they provide is exactly equal to the
conjectured bound.  To have a rigorous proof, it is necessary to
find solution matrices $Q_i$ whose entries are known exactly - they
will ideally be rational. It can then be independently verified that
these matrices satisfy the conditions necessary to prove the desired
result.  We now outline some of the steps that can be taken to find
such a solution.

\newpage

\noindent {\sc Finding a verifiable solution:}

Typically, the space of solutions will be a high-dimensional space,
with many degrees of freedom for the entries of the matrices $Q_i$.
To try to force the solution towards rational entries, we seek to
reduce the dimension of the search space.  There are three methods
we can apply: reducing the size of the block variables, identifying
natural eigenvectors, and changing the basis to introduce
zero-entries.

Recall that for each type $\sigma_i$ we have the associated block variable $Q_i$.
In identifying which types to use, we added all possible types until
we obtained the right bound.  However, it is possible, and even
likely, that some of the types are unnecessary.  Given a type
$\sigma$, we remove it from the SDP problem, and run the SDP solver
again.  If we still obtain the correct bound, then we know the type
$\sigma$ was unnecessary.  If instead this results in a worse bound,
then we keep $\sigma$, and try removing a different type.  In this
way we arrive at a minimal set of necessary types, thus reducing the
number of block variables in the SDP problem.

Given a set of minimal types, there is a further reduction possible.
Every type $\sigma$ has the natural group $\Gamma_{\sigma}$ of
automorphisms of the underlying graph $\sigma_0$.  The group
$\Gamma_{\sigma}$ acts on the algebra $\mathcal{A}^{\sigma}$ by
relabeling the flags according to the automorphism.  We can then
decompose $\mathcal{A}^{\sigma} = \mathcal{A}_+^{\sigma} \oplus
\mathcal{A}_-^{\sigma}$ into a positive and negative part, where
$\mathcal{A}_+^{\sigma}$ consists of all elements invariant under
$\Gamma_{\sigma}$, while $\mathcal{A}_-^{\sigma} \stackrel{def}{=}
\left \{ f \in \mathcal{A}^{\sigma} : \sum_{\gamma \in
\Gamma_{\sigma}} \gamma f = 0 \right \}$.  For example, given the
type and flags of Figure 3, both labelings of the vertices of
$\sigma$ give rise to automorphisms, and so $\Gamma_{\sigma}$ is the
symmetric group on two elements.  One can verify that $F_3 \in
\mathcal{A}_+^{\sigma}$, $F_1 + F_2 \in \mathcal{A}_+^{\sigma}$, and
$F_1 - F_2 \in \mathcal{A}_-^{\sigma}$.

\begin{figure}[H]
\centering
\parbox{1in}{
\centering
\begin{tikzpicture}
  [scale=0.6,auto=left,every node/.style={circle, draw, fill=black!50,inner sep=0pt, minimum width=4pt}]
  \node (n1) at (180:1 cm) [label=left:$1$]{};
  \node (n2) at (360:1 cm) [label=right:$2$]{};

\end{tikzpicture}
\caption*{$\sigma$}
}
\parbox{1in}{
\centering
\begin{tikzpicture}
  [scale=0.6,auto=left,every node/.style={circle, draw, fill=black!50,inner sep=0pt, minimum width=4pt}]
  \node (n1) at (210:1 cm) [label=left:$1$]{};
  \node (n2) at (330:1 cm) [label=right:$2$]{};
  \node (n3) at (90:1 cm) {};

  \foreach \from/\to in {n1/n3}
    \draw (\from) -- (\to);

\end{tikzpicture}
\caption*{$F_1$}
}
\parbox{1in}{
\centering
\begin{tikzpicture}
  [scale=0.6,auto=left,every node/.style={circle, draw, fill=black!50,inner sep=0pt, minimum width=4pt}]
  \node (n1) at (210:1 cm) [label=left:$1$]{};
  \node (n2) at (330:1 cm) [label=right:$2$]{};
  \node (n3) at (90:1 cm) {};

  \foreach \from/\to in {n2/n3}
    \draw (\from) -- (\to);

\end{tikzpicture}
\caption*{$F_2$}
}
\parbox{1in}{
\centering
\begin{tikzpicture}
  [scale=0.6,auto=left,every node/.style={circle, draw, fill=black!50,inner sep=0pt, minimum width=4pt}]
  \node (n1) at (210:1 cm) [label=left:$1$]{};
  \node (n2) at (330:1 cm) [label=right:$2$]{};
  \node (n3) at (90:1 cm) {};

  \foreach \from/\to in {n1/n3,n2/n3}
    \draw (\from) -- (\to);

\end{tikzpicture}
\caption*{$F_3$} } \caption{Decomposition into positive and negative
parts.}

\label{posneg_flags}
\end{figure}

This decomposition is useful because whenever we have $f \in
\mathcal{A}_+^{\sigma}$ and $g \in \mathcal{A}_-^{\sigma}$, we have
$[[ f \cdot g ]]_{\sigma} = 0$.  Hence given the semi-definite
matrix $Q$ for the type $\sigma$, we can split it into its
`invariant' part $Q^+$ and `anti-invariant' part $Q^-$.  While this
increases the number of block variables, they are now of smaller
size, and hence have fewer degrees of freedom, reducing the
dimension of the search space. Moreover, it may be that not all of
these parts are necessary, so we can proceed as before to remove any
unnecessary block variables.

\medskip

The second technique we use is that of identifying natural
eigenvectors. For this, we require an extremal construction that
attains the conjectured bound; let $G_n$ represent an extremal graph
on $n$ vertices, and let $\{ G_n \}_{n \in \mathbb{N}}$.  Given a type $\sigma$, fix a position of $\sigma$ in $G_n$.  This turns $G_n$ into a $\sigma$-flag $F_n$.  The family $\{ F_n \}_{n \in \mathbb{N}}$ represents a way to consistently label the type $\sigma$ in $G_n$.

Recall that in the flag algebra calculations, we used the bound $[[Q]]_{\sigma}(G_n) \ge o_{n \rightarrow \infty}(1)$.  If $G_n$ is an extremal graph, then the bounds are tight, and so $[[Q]]_{\sigma}(G_n) = o_{n \rightarrow \infty}(1)$.  Hence we must have $p_{\sigma}(Q \{ \mathcal{F}_l^{\sigma} \}; F_n) = \sum_{F_1, F_2 \in \mathcal{F}_l^{\sigma}} Q_{F_1,F_2} p_{\sigma}(F_1; F_n) p_{\sigma}(F_2; F_n) + O( 1 / n ) = o_{n \rightarrow \infty}(1)$.  Taking the limit as $n \rightarrow \infty$, this implies that if we have a vector $v$ defined by $v_F = \lim_{n \rightarrow \infty} p_{\sigma}(F ; F_n)$ for $F \in \mathcal{F}_l^{\sigma}$, then $v_F$ must be a zero-eigenvector of $Q$.  Repeating this for different embeddings of the type $\sigma$ in the extremal family of graphs $\{ G_n \}$ can give rise to several eigenvectors.  This procedure is formally defined using the apparatus of \emph{ensembles of random homomorphisms} in Section 3.2 of \cite{raz_flag}.

Having fixed this eigenvectors, we can then reduce the size of the block variables.  Note that if we are able to remove all zero-eigenvectors this way, then we are left with positive definite matrices as our block variables.  This leaves a little room for error, so we can replace the entries with simple rational entries and hope to still have a positive semi-definite matrix.

\medskip

Our final method for reducing the dimension of the search space is to change the basis to introduce zero entries.  Ideally the new set of variables will be a rational linear combination of the previous set, which will lead to a solution with rational entries.  Moreover, we introduce zeros in such a way as to split the block variables into smaller blocks.  More formally, consider the general SDP problem of the following form:

\medskip

maximize $\mathrm{tr}(CX)$, subject to
\begin{itemize}
    \item $\mathrm{tr}(A_iX) = a_i$ for $i = 1, 2, \hdots, m$
    \item $X \succeq 0$ (that is, $X$ is positive semi-definite)
\end{itemize}
where $X$ and $A_i$ are symmetric $n \times n$ matrices for $i = 1, 2, \hdots, m$.

\medskip

Suppose we had a rational $n \times n$ matrix $M$ such that all entries of the first row (and hence column, by symmetry) of $M X M^T$, except possibly the first, were zero.  We can then change variables to modify the SDP problem into an equivalent one, as below:

\medskip

maximize $\mathrm{tr}(\tilde{C}Y)$, subject to
\begin{itemize}
    \item $\mathrm{tr}(\tilde{A_i}Y) = a_i$ for $i = 1, 2, \hdots, m$
    \item $Y \succeq 0$
\end{itemize}
where $\tilde{C} = (M^{-1})^T C M^{-1}$ and $\tilde{A_i} = (M^{-1})^T A_i M^{-1}$ for $i = 1, 2, \hdots, m$.

\medskip

The solutions of both problems are related by the equation $Y = M X M^T$.  We can now reduce the dimension of the solution space by forcing all the non-principle entries of the first row/column of $\tilde{C}$ and $\tilde{A_i}$ to be zero for $i = 1, 2, \hdots, m$.  This is possible because we already have the existence of a solution $Y$ with $Y_{1,j} = Y_{j,1} = 0$ for $j = 2, 3, \hdots, n$, and hence this restricted solution space contains a solution to the original problem.  This operation splits the block variable $Y$ into a one-dimensional block and an $(n-1)$-dimensional block.  We can now iterate the procedure.

We find such a matrix $M$ by inspecting the numerical solution to the original SDP problem, and using a rational approximation to an eigenvector $v$ for the first row.  We then fill in the remaining rows with independent vectors orthogonal to $v$.  Note that if the solution is initially positive definite, there is a little room for error, so we may hope to choose a simple rational approximation without worsening the solution to the SDP problem.

\medskip

\noindent {\sc Expressing the solution as a sum of squares:}

If we are able to repeatedly iterate the change of basis procedure
outlined above, then we will eventually reach a problem whose
solution is a diagonal matrix.  This is advantageous for two
reasons.  First, the semi-definite programming problem reduces to a
linear programming (LP) problem.  This can be solved by only taking
rational linear combinations of the entries of the variables at
every step, and so the solution will be a rational combinations of
the input to the LP problem.  Hence the solution can be specified
exactly, resulting in a verifiable proof.  Second, we can write the
positive semi-definite matrix as a sum of squares, which is easier
to understand.  This can lead to combinatorial interpretations of
the proof, as we demonstrated in Section \ref{43_flags}.  Thus while
this step is not necessary for solving problems with the machinery
of flag algebras, it makes the resulting proofs much more
understandable.

\section{Integer optimization problem} \label{app_intopt}

In this appendix, we prove Lemma \ref{c5opt} from Section \ref{c5stab}, in which we solve the integer optimization problem required to determine the size of the parts in the blow-up of $C_5$ that minimizes the number of $4$-cliques.

\begin{lemma}
Let $\varepsilon > 0$ be sufficiently small, and $n$ sufficiently large.  Consider the function 
\[ g(y_1, y_2, y_3, y_4, y_5) = \sum_{i=1}^5 \binom{y_i}{5} - \sum_{i=1}^5 \binom{n - y_i - y_{i+1}}{4}. \]
Subject to the constraints that the $y_i$ be integers satisfying $\sum_{i=1}^5 y_i = 2n$ and $\left| y_i - \frac{2}{5} n \right| < \varepsilon n$, $g$ is uniquely (up to cyclic permutation of the variables) minimized when the $y_i$ take values $\left \lfloor \frac{2n}{5} \right \rfloor$ and $\left \lceil \frac{2n}{5} \right \rceil$ in ascending order.
\end{lemma}

\begin{proof}
First we will show that if $(y_1, y_2, y_3, y_4, y_5)$ is optimal, the $y_i$ should be as equal as possible.  Suppose towards contradiction that this was not the case.  Then there are $i,j$ with $y_i - y_j \ge 2$; let $i,j$ be such that this difference is maximal over all such pairs.  There are two cases:

\medskip

\underline{Case 1:} $i$ and $j$ are consecutive.

Without loss of generality, suppose $i = 2$ and $j = 3$, so we have $y_2 - y_3 \ge 2$, with this difference being maximal.  We will show that $g(y_1, y_2 - 1, y_3 + 1, y_4, y_5) < g(y_1, y_2, y_3, y_4, y_5)$, which contradicts our assumption of optimality.  Indeed, we have
\begin{align*}
    \Delta g &= g(y_1, y_2 - 1, y_3 + 1, y_4, y_5) - g(y_1, y_2, y_3, y_4, y_5) \\
    &= \binom{y_2-1}{4} + \binom{y_3 + 1}{4} - \binom{n - y_1 - y_2 + 1}{4} + \binom{n - y_3 - y_4 - 1}{4} \\
    & \quad - \left[ \binom{y_2}{4} + \binom{y_3}{4} - \binom{n - y_1 - y_2}{4} - \binom{n - y_3 - y_4}{4} \right] \\
    &= \binom{y_3}{3} - \binom{y_2 - 1}{3} + \binom{n - y_3 - y_4 - 1}{3} - \binom{n - y_1 - y_2}{3}.
\end{align*}

Now let $s = y_2 - y_3 - 1 \ge 1$, and let $t = (n - y_3 - y_4 - 1) - (n - y_1 - y_2) = y_1 - y_4 + y_2 - y_3 - 1 = y_1 - y_4 + s$.  If $t \le 0$, then clearly the above expression is negative, which shows $(y_1, y_2, y_3, y_4, y_5)$ is not optimal.  Hence we must have $t \ge 1$.  In this case, we can rewrite the above as
\[ \Delta g = \left[ \binom{t}{3} + \binom{t}{2} ( n - y_1 - y_2 ) + t \binom{n - y_1 - y_2}{2} \right] - \left[ \binom{s}{3} + \binom{s}{2} y_3 + s \binom{y_3}{2} \right]. \]

From our constraints on the variables $y_i$, we have that $y_3 = \left( \frac{2}{5} + O(\varepsilon) \right)n$, $n - y_1 - y_2 = \left( \frac{1}{5} + O( \varepsilon ) \right)n$, $s \le 2\varepsilon n$ and $t \le 4 \varepsilon n$.  These bounds imply that the main terms are those linear in $s$ and $t$.  We have
\[ \Delta g = \frac{1}{50} \left[ \left( 1 + O( \varepsilon ) \right) t - \left( 4 + O( \varepsilon ) \right) s \right] n^2 + O ( (s^2 + t^2) n ). \]
In particular, for large $n$, this can only be non-negative if $t \ge \left( 4 - O( \varepsilon) \right) s$.  However, we have $t = y_1 - y_4 + s$, and by our assumption of maximality of $y_2 - y_3$, we have $y_1 - y_4 \le y_2 - y_3 = s + 1$.  Hence $t \le 2s + 1$, and we have a contradiction.

\medskip

\underline{Case 2:} $i$ and $j$ are not consecutive.

Without loss of generality, suppose $i = 2$ and $j = 4$, with $y_2 - y_4 \ge 2$ being the maximal difference.  Let
\[ \Delta g = g(y_1, y_2 - 1, y_3, y_4 + 1, y_5) - g(y_1, y_2, y_3, y_4, y_5). \]
By similar calculations to those in Case 1, we have
\[ \Delta g = \binom{y_4}{3} - \binom{y_2 - 1}{3} + \binom{n - y_5 - y_4 - 1}{3} + \binom{n - y_4 - y_3 - 1}{3} - \binom{n - y_2 - y_3}{3} - \binom{n - y_1 - y_2}{3}. \]
We define $s = y_2 - y_4 - 1$, and $t = (n - y_5 - y_4 - 1) - (n - y_1 - y_2) = y_1 - y_5 + s$.  If $t \le 0$, then $\Delta g < 0$, which contradicts the optimality of $(y_1, y_2, y_3, y_4, y_5)$.  Hence we may assume $t \ge 1$, and rewrite $\Delta g$ in terms of $s$ and $t$ as before.  In this case we find
\[ \Delta g = \frac{1}{50} \left[ \left( 1 + O( \varepsilon ) \right) t - \left( 3 + O( \varepsilon ) \right) s \right] n^2 + O ( (s^2 + t^2) n ). \]
Hence for $\Delta g \ge 0$, we must have $t \ge \left(3 - O( \varepsilon) \right) s$.  However, by maximality of $y_2 - y_4$, we have $t = y_1 - y_5 + s \le 2s + 1$.  The only way these equations can be satisfied is if $s =1$ and $y_1 - y_5 = 2$.  But in this case $y_1$ and $y_5$ are two consecutive variables with a maximal difference, and so we reduce to Case 1, which leads to a contradiction.

\medskip

Hence we have shown that subject to the above conditions, $g$ is only minimised when the variables $y_i$ take values $\left \lfloor \frac{2n}{5} \right \rfloor$ or $\left \lceil \frac{2n}{5} \right \rceil$.  If $n \equiv 0, 1, 4 \pmod 5$, there is only one way (up to cyclic rotation) that these values can be distributed, so the minimum is uniquely determined.  If $n \equiv 2, 3 \pmod 5$, then there are two possible distributions of the values.  In each case, an easy calculation shows $g$ is minimised when the values are in decreasing order.  This completes the proof of the lemma.
\end{proof}

\medskip

Note that we assume $|y_i - \frac{2}{5} n | < \varepsilon n$ only to simplify the proof.  Even without this condition, we can prove that for any $n \ge 12$, the above result holds.  However, as the flag algebra results are asymptotic in nature, we can only determine the unique extremal graph for the $(4,3)$-problem when $n$ is large.

\end{document}